%% file: Matrices6.tex
\documentclass[titlepage,a4paper,11pt,twoside]{article}

\usepackage[right=3cm,left=3cm,top=3cm,bottom=3cm]{geometry}
\usepackage{amsfonts}
\usepackage{amsthm}
\usepackage[latin1]{inputenc}
\usepackage{amssymb, amsmath}
\usepackage[automark,plainheadtopline,plainheadsepline,headtopline,nouppercase]{scrpage2}
\usepackage{graphicx}
\usepackage{color}

\numberwithin{equation}{section}

\setheadsepline{.5pt}
\automark[section]{section}
\lehead[]{Functional Analysis I}
\rehead[]{}
\lohead[]{}
\rohead[]{\headmark}
\newtheorem{theo}{Theorem}[section]

\newtheorem{lemma}[theo]{Lemma}

\theoremstyle{definition} 

\newtheorem{rem}[theo]{Remark}
\newtheorem{definition}[theo]{Definition}

\newcommand{\R}{\mathbb{R}}
\newcommand{\N}{\mathbb{N}}
\newcommand{\C}{\mathbb{C}}

\renewcommand{\P}{\mathbb{P}}
\newcommand{\E}{\mathbb{E}}
\newcommand{\X}{\mathcal{X}}
\newcommand{\conv}{\mathrm{conv}}
\newcommand{\rank}{\mathrm{rank}}
\newcommand{\linspan}{\mathrm{span}}
\newcommand{\diag}{\mathrm{diag}}
\newcommand{\kernel}{\mathrm{kern}}
\newcommand{\argmin}{\operatornamewithlimits{arg\,min}}
\newcommand{\tr}{\mathrm{tr}}
\newcommand{\vol}{\mathrm{vol}}
\newcommand{\dist}{\mathrm{dist}}
\newcommand{\sgn}{\mathrm{sgn}}
\newcommand{\range}{\mathrm{range}}

\newcommand{\leM}{\preccurlyeq}
\newcommand{\geM}{\succcurlyeq}

\newlength{\fixboxwidth}
\makeindex
\begin{document}
\pagenumbering{alph}
\pagestyle{empty}
\begin{titlepage}
\title{{\bf Random Matrices}\\[.5cm]{\bf and Matrix Completion}\\[2cm]\uppercase{\small {Lecture Script}}}
\author{Jan Vybiral\\[10cm]}
\date{\today}
\maketitle
\end{titlepage}

\tableofcontents
\newpage

\pagenumbering{arabic}
\setcounter{page}1
\pagestyle{scrheadings}
\thispagestyle{scrplain}

\setcounter{page}{1}
\lehead[]{Random Matrices and Matrix Completion}
\section*{Foreword}

The aim of this note (as well as of the course itself) is to give
a largely self-contained proof of two of the main results in the field
of low-rank matrix recovery. This field aims for identification
of low-rank matrices from only limited linear information
exploiting in a crucial way their very special structure.
As a crucial tool we develop also the basic statements
of the theory of random matrices.

The notes are based on a number of sources, which appeared in the last few years.
As we give only the minimal amount of the subject needed for the application
in mind, the reader is invited to study this further reading in detail.

\newpage

\section{Introduction to randomness}

Before we come to the main subject of our work, we give a brief introduction to
the role of randomness in functional analysis and numerics. Although some of the results
presented here are not used later on in the text, the methods used here already introduce
some of the main ideas.

\subsection{Approximate Caratheodory theorem}

Classical Caratheodory's theorem states that a point in a convex hull of any set in $\R^n$
is actually also a convex combination of only $n+1$ points from this set.

\begin{theo}(Caratheodory's theorem). Consider a set $A$ in $\R^n$ and a point $x\in\conv(A)$.
Then there exists a subset $A_0\subset A$ of cardinality $|A_0|\le  n + 1$ such that $x\in\conv(A_0)$. In other words,
every point in the convex hull of $A$ can be expressed as a convex combination of at most $n+1$ points from $A$.
\end{theo}

We will show a dimension-independent approximative version of this theorem. The proof is probabilistic - the existence of a good
linear combination is proven to exist by estimating a mean of certain random variables.
As they can take only finitely many values, no extensive introduction into probability theory is needed.

We will need a notion of a radius of a set in a Hilbert space, which is given simply by
$$
r(A)=\sup\{\|a\|:a\in A\}.
$$

\begin{theo}(Approximate Caratheodory's theorem). Consider a bounded set $A$ in a Hilbert
space $H$ and a point $x\in\conv(A)$. Then, for every $N\in\N$, one can find points $x_1,\dots,x_N \in A$ such
that
$$
\Bigl\|x-\frac{1}{N}\sum_{j=1}^Nx_j\Bigr\|\le \frac{r(A)}{\sqrt{N}}.
$$
\end{theo}
\begin{proof}Let $x\in\conv(A)$. Then it can be written as a convex combination of some points $z_1,\dots,z_m\in A$
with coefficients $\lambda_1,\dots,\lambda_m\ge 0$, $\lambda_1+\dots+\lambda_m=1$:
$$
x=\sum_{j=1}^m \lambda_j z_j.
$$
Let us now consider a random vector-valued variable $Z$ with values in $H$, which takes the value $z_j$ with probability $\lambda_j.$
Then
$$
\E\,Z=\sum_{j=1}^m\lambda_jz_j=x.
$$
In other words, on average, the value of $Z$ is $x$. On the other hand, if $Z_1, Z_2,\dots$ are independent copies of $Z$, then
$\frac{1}{N}\sum_{j=1}^N Z_j$ should tend to the mean of $Z$ as $N\to \infty.$ Indeed, in the mean we have
\begin{align*}
\E\Bigl\|x-\frac{1}{N}\sum_{j=1}^N Z_j\Bigr\|^2&=\E\Bigl\langle x-\frac{1}{N}\sum_{j=1}^N Z_j, x-\frac{1}{N}\sum_{k=1}^N Z_k\Bigr\rangle\\
&=\|x\|^2-\frac{2}{N}\E\Bigl\langle x,\sum_{j=1}^N Z_j\Bigr\rangle+\frac{1}{N^2}\sum_{j,k=1}^N\E \langle Z_j,Z_k \rangle\\
&=-\|x\|^2+\frac{1}{N^2}\sum_{j,k=1}^N\E \langle Z_j,Z_k \rangle.
\end{align*}
If $j=k$, the pair $(Z_j,Z_j)$ takes values $(z_j,z_j)$ with probability $\lambda_j$ and
$$
\E\langle Z_j,Z_j\rangle=\sum_{l=1}^m\lambda_l \langle z_l,z_l\rangle=\sum_{l=1}^m\lambda_l \|z_l\|^2\le r(A)^2.
$$
If $j\not= k$, the independence of $Z_j$ and $Z_k$ shows that the pair $(Z_j,Z_k)$ takes the value $(z_l,z_{l'})$ with probability $\lambda_l\cdot\lambda_{l'}$
and
$$
\E\langle Z_j,Z_k\rangle =\sum_{l,l'=1}^m \lambda_l\lambda_{l'}\langle z_l,z_{l'}\rangle=\|x\|_2^2.
$$
Finally,
\begin{align*}
\E\Bigl\|x-\frac{1}{N}\sum_{j=1}^N Z_j\Bigr\|^2&=-\|x\|^2+\frac{1}{N^2}\sum_{j,k=1}^N\E \langle Z_j,Z_k \rangle\\
&=-\|x\|^2+\frac{1}{N^2}\sum_{j=1}^N\E \langle Z_j,Z_j \rangle+\frac{1}{N^2}\sum_{j\not=k}\E\langle Z_j,Z_k\rangle\\
&\le -\|x\|^2+\frac{Nr(A)^2}{N^2}+\frac{N(N-1)}{N^2}\|x\|^2\\
&=\frac{r(A)^2}{N}-\frac{\|x\|^2}{N}\le\frac{r(A)^2}{N}.
\end{align*}
There is therefore a realization of the random variables $Z_i$ (i.e. one point $\omega$ in the probability space), such that
$$
\Bigl\|x-\frac{1}{N}\sum_{j=1}^N Z_j(\omega)\Bigr\|\le\frac{r(A)}{\sqrt{N}}.
$$
Putting $x_j=Z_j(\omega)$, we finish the proof.

\end{proof}

\subsection{Monte Carlo integration}

The use of random constructions and algorithms became a standard technique in the last decades in many different areas of mathematics.
As one example out of many let us sketch their use in numerical integration. Let us assume that we have a function $f:\Omega_d\to \R$,
where $\Omega_d\subset\R^d$ has (for simplicity) measure 1. We would like to approximate the integral of $f$
$$
I=\int_{\Omega_d}f(x)dx
$$
using only a limited number of function values of $f$. The methods of Monte Carlo propose to replace the classical cubature formulas
(which typically scale badly with $d\to\infty$) by a sum
$$
I(x_1,\dots,x_n)=\frac{1}{n}\sum_{j=1}^n f(x_j),
$$
where $x_j$'s are chosen independently and randomly from $\Omega_d.$ It is easy to see, that on the average we have indeed
$$
\E I(x_1,\dots,x_n)=\frac{1}{n}\sum_{j=1}^n \E f(x_j)=I.
$$
But we are of course also interested how much do $I$ and $I(x_1,\dots,x_n)$ differ for some choice of $x_1,\dots,x_n$,
i.e. how big is $|I-I(x_1,\dots,x_n)|$. If we measure this error in the $L_2$-sense, we obtain easily
\begin{align*}
\E|I-I(x_1,\dots,x_n)|^2&=\E[I^2-2I\cdot I(x_1,\dots,x_n)+I^2(x_1,\dots,x_n)]\\
&=I^2-2I\cdot \E I(x_1,\dots,x_n)+\E I^2(x_1,\dots,x_n)\\
&=-I^2+\E I^2(x_1,\dots,x_n)\\
&=-I^2+\frac{1}{n^2}\sum_{j=1}^n\E f(x_j)^2+\frac{1}{n^2}\sum_{j\not=k}\E f(x_j)f(x_k)\\
&=-I^2+\frac{\|f\|_2^2}{n}+\frac{n(n-1)}{n^2}I^2=\frac{\|f\|_2^2}{n}-\frac{1}{n}I^2\le \frac{\|f\|_2^2}{n}.
\end{align*}
Hence
$$
\Bigl(\E|I-I(x_1,\dots,x_n)|^2\Bigr)^{1/2}\le \frac{\|f\|_2}{\sqrt{n}}
$$
independently on $d$ and the regularity properties of $f$.
\subsection{Concentration of measure}

If $\omega_1,\dots,\omega_m$ are (possibly dependent) standard normal random variables, then ${\mathbb E} (\omega_1^2+\dots+\omega_m^2)=m$. If $\omega_1,\dots, \omega_m$ are even independent,
then the value of $\omega_1^2+\dots+\omega_m^2$ concentrates very strongly around $m$. This effect is known as \emph{concentration of measure}, cf. \cite{Led, LT, MS}.
Before we come to a quantitative description of this effect, we need two simple facts about standard normal variables.

\begin{lemma}\label{lem:N}
\begin{enumerate}
\item[(i)] Let $\omega$ be a standard normal variable. Then ${\mathbb E}\,(e^{\lambda \omega^2})=1/\sqrt{1-2\lambda}$ \ for\  $-\infty<\lambda<1/2.$
\item[(ii)] (2-stability of the normal distribution) Let $m\in{\mathbb N}$, let $\lambda=(\lambda_1,\dots,\lambda_m)\in\R^m$
and let $\omega_1,\dots,\omega_m$ be i.i.d. standard normal variables. 
Then $\lambda_1\omega_1+\dots+\lambda_m\omega_m \sim (\sum_{i=1}^m \lambda_i^2)^{1/2}\cdot{\mathcal N}(0,1)$, i.e. it is equidistributed with a multiple of
a standard normal variable.
\end{enumerate}
\end{lemma}
\begin{proof}
The proof of (i) follows from the substitution $s:=\sqrt{1-2\lambda}\cdot t$ in the following way.
\begin{align*}
{\mathbb E}\,(e^{\lambda\omega^2})=\frac{1}{\sqrt{2\pi}}\int_{-\infty}^\infty e^{\lambda t^2}\cdot e^{-t^2/2}dt &= 
\frac{1}{\sqrt{2\pi}}\int_{-\infty}^\infty e^{(\lambda-1/2) t^2}dt\\&= 
\frac{1}{\sqrt{2\pi}}\int_{-\infty}^\infty e^{- s^2/2}\cdot \frac{ds}{\sqrt{1-2\lambda}} = \frac{1}{\sqrt{1-2\lambda}}.
\end{align*}
Although the property (ii) is very well known (and there are several different ways to prove it), we provide a simple geometric proof for the sake of completeness.
It is enough to consider the case $m=2$. The general case then follows by induction. 

Let therefore $\lambda=(\lambda_1,\lambda_2)\in\R^2, \lambda\not=0$, be fixed and let $\omega_1$ and $\omega_2$ be i.i.d. standard normal random variables.
We put $S:=\lambda_1\omega_1+\lambda_2\omega_2.$ Let $t\ge 0$ be an arbitrary non-negative real number. We calculate
\begin{align*}
{\mathbb P}(S\le t)&=\frac{1}{2\pi}\int_{(u,v):\lambda_1u+\lambda_2v\le t}e^{-(u^2+v^2)/2}dudv
=\frac{1}{2\pi}\int_{u\le c; v\in\R}e^{-(u^2+v^2)/2}dudv\\
&=\frac{1}{\sqrt{2\pi}}\int_{u\le c}e^{-u^2/2}du.
\end{align*}
We have used the rotational invariance of the function $(u,v)\to e^{-(u^2+v^2)/2}$. The value
of $c$ is given by the distance of the origin from the line $\{(u,v):\lambda_1u+\lambda_2v=t\}$.
It follows by elementary geometry and Pythagorean theorem that (cf. $\Delta OAP\simeq\Delta BAO$ in Figure \ref{fig:c})
$$
c=|OP|=|OB|\cdot\frac{|OA|}{|AB|}
=\frac{t}{\sqrt{\lambda_1^2+\lambda_2^2}}.
$$
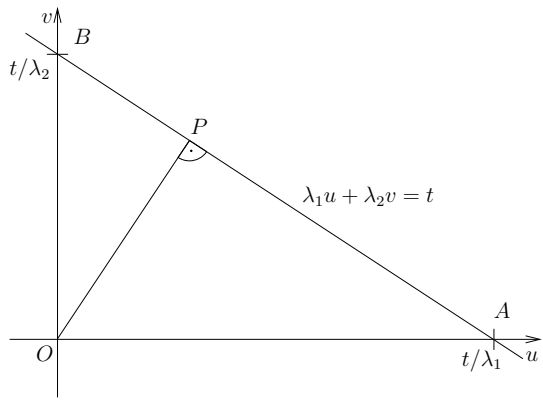
\begin{figure}[h]
\centering
\scalebox{.7}{\input{fig_c.pstex_t}}
\caption{Calculating $c=|OP|$ by elementary geometry for $\lambda_1,\lambda_2>0$}\label{fig:c}
\end{figure}

We therefore get
\begin{align*}
{\mathbb P}(S\le t)=\frac{1}{\sqrt{2\pi}}\int_{\sqrt{\lambda_1^2+\lambda_2^2}\cdot u\le t}e^{-u^2/2}du
={\mathbb P}\left(\sqrt{\lambda_1^2+\lambda_2^2}\cdot \omega\le t\right).
\end{align*}
The same estimate holds for negative $t$'s by symmetry and the proof is finished.
\end{proof}

Following lemma is the promised description of concentration of $\omega_1^2+\dots+\omega_m^2$ around its mean, i.e. $m$.
It shows that the probability, that $\omega_1^2+\dots+\omega_m^2$ is much larger (or much smaller) than $m$ is \emph{exponentially} small!

\begin{lemma}\label{Lemma:Con}
Let $m\in{\mathbb N}$ and let $\omega_1,\dots,\omega_m$ be i.i.d. standard normal variables. Let $0<\varepsilon<1$. Then
\begin{equation*}
{\mathbb P}(\omega_1^2+\dots+\omega_m^2\ge(1+\varepsilon)m)\le e^{-\frac{m}{2}[\varepsilon^2/2-\varepsilon^3/3]}
\end{equation*}
and
\begin{equation*}
{\mathbb P}(\omega_1^2+\dots+\omega_m^2\le(1-\varepsilon)m)\le e^{-\frac{m}{2}[\varepsilon^2/2-\varepsilon^3/3]}.
\end{equation*}
\end{lemma}
\begin{proof}
We prove only the first inequality. The second one follows in exactly the same manner. Let us put $\beta:=1+\varepsilon>1$ and calculate
\begin{align*}
{\mathbb P}(\omega_1^2+\dots+\omega_m^2\ge\beta m)&= {\mathbb P}(\omega_1^2+\dots+\omega_m^2-\beta m\ge 0) \\
&= {\mathbb P}(\lambda(\omega_1^2+\dots+\omega_m^2-\beta m)\ge 0)\\
&={\mathbb P}(\exp(\lambda(\omega_1^2+\dots+\omega_m^2-\beta m))\ge 1)\\
&\le{\mathbb E} \exp(\lambda(\omega_1^2+\dots+\omega_m^2-\beta m)),
\end{align*}
where $\lambda>0$ is a positive real number, which shall be chosen later on. We have used the Markov's inequality 
in the last step. Further we use the elementary properties of exponential function and 
the independence of the variables $\omega_1,\dots,\omega_m$. This leads to
\begin{align*}
{\mathbb E} \exp(\lambda(\omega_1^2+\dots+\omega_m^2-\beta m))&=e^{-\lambda\beta m}\,\cdot\,{\mathbb E}\,e^{\lambda \omega_1^2}\cdots e^{\lambda \omega_m^2}
=e^{-\lambda\beta m}\,\cdot\,({\mathbb E}\,e^{\lambda \omega_1^2})^m
\end{align*}
and with the help of Lemma \ref{lem:N} we get finally (for $0<\lambda<1/2$)
\begin{align*}
{\mathbb E} \exp(\lambda(\omega_1^2+\dots+\omega_m^2-\beta m))&=e^{-\lambda\beta m}\cdot (1-2\lambda)^{-m/2}.
\end{align*}
We now look for the value of $0<\lambda<1/2$, which would minimize the last expression.
Therefore, we take the derivative of $e^{-\lambda\beta m}\cdot (1-2\lambda)^{-m/2}$ and put it equal to zero.
After a straightforward calculation, we get
\[
\lambda=\frac{1-1/\beta}{2},
\]
which obviously satisfies also $0<\lambda<1/2$.
Using this value of $\lambda$ we obtain
\begin{align*}
{\mathbb P}(\omega_1^2+\dots+\omega_m^2\ge \beta m)&\le e^{-\frac{1-1/\beta}{2}\cdot\beta m}\cdot (1-(1-1/\beta))^{-m/2}
=e^{-\frac{\beta-1}{2}m}\cdot \beta^{m/2}\\
&=e^{-\frac{\varepsilon m}{2}}\cdot e^{\frac{m}{2}\ln (1+\varepsilon)}.
\end{align*}
The result then follows from the inequality
$$
\ln(1+t)\le t-\frac{t^2}{2}+\frac{t^3}{3},\quad -1<t<1.
$$
\end{proof}

\subsection{Lemma of Johnson-Lindenstrauss}

The effect of Concentration of measure has far reaching consequences. We will present only one of them, called Lemma of Johnson and Lindenstrauss.

We denote until the end of this section
\begin{equation}\label{eq:defA}
A=\frac{1}{\sqrt m}\left(\begin{array}{ccc}\omega_{1,1}& \dots & \omega_{1n}\\
\vdots&\ddots&\vdots\\\omega_{m1}&\dots&\omega_{mn}\end{array}\right),
\end{equation}
where $\omega_{ij}, i=1,\dots,m, j=1,\dots,n$, are i.i.d. standard normal variables.

Using 2-stability of the normal distribution, Lemma \ref{Lemma:Con} shows immediately that $A$ defined as in \eqref{eq:defA}
acts with high probability as isometry on one fixed $x\in\R^n.$

\begin{theo}\label{thm:conx} Let $x\in \R^n$ with $\|x\|_2=1$ and let $A$ be as in \eqref{eq:defA}. Then 
\begin{equation}\label{eq:con1}
{\mathbb P}\Bigl(\Bigl|\|Ax\|_2^2-1\Bigr|\ge t\Bigr)\le 2e^{-\frac{m}{2}[t^2/2-t^3/3]}\le 2e^{-Cmt^2}
\end{equation}
for $0<t< 1$ with an absolute constant $C>0.$
\end{theo}
\begin{proof}
Let $x=(x_1,x_2,\dots,x_n)^T$. Then we get by the 2-stability of normal distribution and Lemma \ref{Lemma:Con}
\begin{align*}
{\mathbb P}&\Bigl(\Bigl|\|Ax\|_2^2-1\Bigr|\ge t\Bigr)\\
&={\mathbb P}\Bigl(\bigl| (\omega_{1,1}x_1+\dots+\omega_{1n}x_n)^2+\dots+(\omega_{m1}x_1+\dots+\omega_{mn}x_n)^2-m\bigr|\ge mt\Bigr)\\
&={\mathbb P}\Bigl(\bigl|\omega_1^2+\dots+\omega_m^2-m\bigr|\ge mt\Bigr)\\
&={\mathbb P}\Bigl(\omega_1^2+\dots+\omega_m^2\ge m(1+t)\Bigr)+{\mathbb P}\Bigl(\omega_1^2+\dots+\omega_m^2\le  m(1-t)\Bigr)\\
&\le 2e^{-\frac{m}{2}[t^2/2-t^3/3]}.
\end{align*}
This gives the first inequality in \eqref{eq:con1}. The second one follows by simple algebraic manipulations (for $C=1/12$).
\end{proof}
\begin{rem} (i) Observe, that \eqref{eq:con1} may be easily rescaled to
\begin{equation}\label{eq:con2}
{\mathbb P}\Bigl(\Bigl|\|Ax\|_2^2-\|x\|^2_2\Bigr|\ge t\|x\|^2_2\Bigr)\le 2e^{-Cmt^2},
\end{equation}
which is true for every $x\in\R^n$.\\
(ii) A slightly different proof of \eqref{eq:con1} is based on the rotational invariance of the distribution underlying the random structure of matrices defined by \eqref{eq:defA}.
Therefore, it is enough to prove \eqref{eq:con1} only for one fixed element $x\in\R^n$ with $\|x\|_2=1$. Taking $x=e_1=(1,0,\dots,0)^T$ to be the first canonical unit vector
allows us to use Lemma \ref{Lemma:Con} without the necessity of applying the 2-stability of normal distribution.
\end{rem}

Lemma of Johnson and Lindenstrauss states that a set of points in a high-dimensional space
can be embedded into a space of much lower dimension in such a way that the mutual distances between the points are nearly preserved.

\begin{lemma}\label{Lemma:JL} (Lemma of Johnson and Lindenstrauss). Let $0<\varepsilon<1$ and let $m,N$ and $n$ be natural numbers with
$$
m\ge 4(\varepsilon^2/2-\varepsilon^3/3)^{-1}\ln N.
$$ 
Then for every set $\{x^1,\dots,x^N\}\subset \R^n$ there exists a mapping $f:\R^n\to \R^m$, such that
\begin{equation}\label{eq:JL}
(1-\varepsilon)\|x^i-x^j\|_2^2\le \|f(x^i)-f(x^j)\|_2^2\le (1+\varepsilon)\|x^i-x^j\|_2^2,\qquad i,j\in\{1,\dots,N\}.
\end{equation}
\end{lemma}
\begin{proof}
We put $f(x)=Ax$, where again
$$
Ax=\frac{1}{\sqrt m}\left(\begin{array}{ccc}\omega_{1,1}& \dots & \omega_{1n}\\
\vdots&\ddots&\vdots\\\omega_{m1}&\dots&\omega_{mn}\end{array}\right)x,
$$
and $\omega_{ij}, i=1,\dots,m, j=1,\dots,n$ are i.i.d. standard normal variables. 
We show that with this choice $f$ satisfies \eqref{eq:JL} with positive probability. This proves the existence of such a mapping.

Let $i,j\in\{1,\dots,N\}$ arbitrary with $x^i\not=x^j.$ Then we put $z=\frac{x^i-x^j}{\|x^i-x^j\|_2}$ and evaluate the probability that 
the right hand side inequality in \eqref{eq:JL} does not hold. Theorem \ref{thm:conx} then implies
\begin{align*}
{\mathbb P}\Bigl(\Bigl|\|f(x^i)-f(x^j)\|_2^2-\|x^i-x^j\|_2^2\Bigr|>\varepsilon \|x^i-x^j\|_2^2\Bigr)&={\mathbb P}\Bigl(\Bigl|\|Az\|^2-1\Bigr|>\varepsilon\Bigr)\\
&\le e^{-\frac{m}{2}[\varepsilon^2/2-\varepsilon^3/3]}.
\end{align*}
The same estimate is also true for all $\binom{N}{2}$ pairs $\{i,j\}\subset \{1,\dots,N\}$ with $i\not=j.$
The probability, that one of the inequalities in \eqref{eq:JL} is not satisfied is therefore at most
$$
2\cdot\binom{N}{2}\cdot e^{-\frac{m}{2}[\varepsilon^2/2-\varepsilon^3/3]}< N^2 \cdot e^{-\frac{m}{2}[\varepsilon^2/2-\varepsilon^3/3]}=
\exp\Bigl(2\ln N-\frac{m}{2}[\varepsilon^2/2-\varepsilon^3/3]\Bigr)\le e^0=1
$$
for $m\ge 4(\varepsilon^2/2-\varepsilon^3/3)^{-1}\ln N$. Therefore, the probability that \eqref{eq:JL} holds for all $i,j\in\{1,\dots,N\}$
is positive and the result follows.
\end{proof}

\newpage

\section{Matrix recovery with RIP}\label{sec:2}

The aim of this section is to show, how randommnes can be used in a problem called ``matrix recovery''.
We start by introducing the problem, recalling few basic facts from the theory of compressed sensing
and giving some elementary notions from linear algebra. In the rest of this section we then
present the most simple (but not the most effective) way of reconstructing a low-rank matrix from
a small number of linear measurements.

\subsection{Introduction}

The aim of this course is to show, how low-rank matrices can be reconstructed from only a limited amount
of (linear) information. The key is to combine in an efficient way the structural assumption on the matrix
with the limited information available. In this aspect, it resembles very much the area of compressed sensing,
from which it indeed profited. Before we formalize the setting of low-rank matrix recovery, we will therefore
describe the basic aspects of compressed sensing. We present only few of the most important results,
the (largely self-contained) proofs may be found for example in \cite{BCKV}.

\subsubsection{Briefly about compressed sensing}

Compressed sensing (in its extremely simplified form) studies underdetermined systems of linear equations $Ax=y$,
where $y\in\R^m$ and $A\in\R^{m\times N}$ are given and we look for the solution $x\in\R^N$.
From linear algebra we know that if $m<N$, this system might have none or many solutions.
The crucial ingredient of compressed sensing (motivated by experience from many different areas of applied science)
is to assume that the unknown solution $x$ is \emph{sparse}, i.e. it has only few non-zero entries. We denote by
\begin{equation}\label{eq:norm0}
\|x\|_0:=|\{j\in\{1,\dots,N\}:x_j\not=0\}|
\end{equation}
the number of non-zero entries of $x$. Furthermore, vector $x\in\R^N$ is called $k$-sparse, if $\|x\|_0\le k.$
Compressed sensing then studies if the equation $Ax=y$ has, for given $A\in\R^{m\times N}$ and $y\in\R^m$,
an $k$-sparse solution $x$, if it is unique, and how can it be found.

Unfortunately, in this general form this problem is NP-hard. Nevertheless, for \emph{some} inputs, i.e. for \emph{some} matrices $A\in\R^{m\times N}$
and \emph{some} right-hand sides $y\in\R^m$, the task can be solved in polynomial time, by $\ell_1$-minimization 

\[
\min_z \|z\|_1\quad \text{s.t.}\quad y=Az. \tag{$P_1$} \label{eq:P_1}
\]

Let us recall that
\begin{equation}
\|x\|_p=\begin{cases}\displaystyle\Bigl(\sum_{j=1}^N|x_j|^p\Bigr)^{1/p}&\quad \text{for}\ 0<p<\infty,\\
\displaystyle\max_{j=1,\dots,N}|x_j|&\quad\text{for}\ p=\infty.
\end{cases}
\end{equation}
The analysis of compressed sensing is nowadays typically performed using two notions, the \emph{Null Space Property}
and the \emph{Restricted Isometry Property}.

Obviously, we can not recovery $x$ from $A$ and $y$ only, if $y=Ax$ is zero. The recovery is therefore impossible,
if sparse vectors lie in the kernel of $A$. Actually, the notion of NSP shows that the recovery is possible, if the vectors
from the kernel of $A$ are well spread and do not attain large values on a small set of indices.

\begin{definition} Let $A\in \R^{m\times N}$ and let $k\in\{1,\dots,N\}$. Then $A$ is said to have the \emph{Null Space Property} (NSP)
of order $k$ if
\begin{equation}\label{eq:def:NSP}
\|v_T\|_1<\|v_{T^c}\|_1\quad \text{for all}\ v\in \ker\, A\setminus\{0\}\ \text{and all}\ T\subset\{1,\dots,N\}\ \text{with}\ |T|\le k.
\end{equation}
\end{definition}

\begin{theo}\label{thm:NSP} Let $A\in\R^{m\times N}$ and let $k\in\{1,\dots,N\}$. Then every $k$-sparse vector $x$ is the unique solution of $(P_1)$ with $y=Ax$ if, and only if,
$A$ has the NSP of order $k$.
\end{theo}

Although quite simple, Theorem \ref{thm:NSP} indeed describes the heart of compressed sensing.
In signal processing we often assume (by our experience or intuition) that the incoming unknown signal $x\in\R^N$ is sparse (or nearly sparse).
Theorem \ref{thm:NSP} then tells that if we use a sensing device which acquires only $m$ linear measurements of $x$, then
we can reconstruct $x$ from $A$ (which we of course know) and the measurements $y$ by a convex optimization problem $(P_1)$.
The necessary and sufficient condition for success is that the (newly designed) sensing device satisfies the NSP.

Although nice in the theory, Theorem \ref{thm:NSP} has one important drawback. For a given matrix $A$ it is not easy to check
if it has NSP, or not. The way out is to show, that if $A$ has another property called RIP (see below) then it has also NSP.

\begin{definition}
Let $A\in\R^{m\times N}$ and let $k\in\{1,\dots,N\}$. Then the \emph{restricted isometry constant} $\delta_k=\delta_k(A)$ of $A$ of order $k$ is the smallest $\delta\ge 0$, such that
\begin{equation}\label{eq:RIP}
(1-\delta)\|x\|_2^2\le \|Ax\|_2^2\le (1+\delta)\|x\|_2^2\quad \text{for all}\ x\in\R^N \ \text{with}\ \|x\|_0\le k.
\end{equation}
Furthermore, we say that $A$ satisfies the \emph{Restricted Isometry Property} (RIP) of order $k$ with the constant $\delta_k$ if $\delta_k<1.$
\end{definition}

If the matrix has RIP, it indeed has also NSP and the sparse recovery by $(P_*)$ succeeds.

\begin{theo}\label{thm:RIPNSP}
Let $A\in\R^{m\times N}$ and let $k$ be a natural number with $k\le N/2.$ If $\delta_{2k}(A)<1/3$, then $A$ has the NSP of order $k$.
\end{theo}

What remains is to discuss how to construct matrices with small RIP constants. Although a huge effort was invested into
the search for deterministic constructions, the most simple and most effective way of producing RIP matrices is to use random matrices.
In the most simple case (which is unfortunately not always suitable for applications) one can draw each entry of a matrix
independently from some fixed distribution. We will describe the analysis in the case of random Gaussian matrices given by

\begin{equation}\label{eq:ndefA}
A=\frac{1}{\sqrt m}\left(\begin{array}{ccc}\omega_{1,1}& \dots & \omega_{1N}\\
\vdots&\ddots&\vdots\\\omega_{m1}&\dots&\omega_{mN}\end{array}\right),
\end{equation}
where $\omega_{ij}, i=1,\dots,m, j=1,\dots,N$, are i.i.d. standard normal variables.

Finally, the following theorem shows that such random matrices indeed satisfy the RIP with high probability
if $m$ is large enough - it grows linearly with the sparsity level $k$, logarithmically with the underlying dimension $N$,
and logarithmically with the desired confidence level $1/\varepsilon$. It is especially the logarithmic dependence on $N$
what makes these results attractive for the analysis of high-dimensional phenomena.

\begin{theo}\label{thm:RIP}
Let $N\ge m\ge k\ge 1$ be natural numbers and let $0<\varepsilon<1$ and $0<\delta<1$ be real numbers with
\begin{equation}\label{eq:tech1}
m\ge C\delta^{-2}\Bigl(k\ln(eN/k)+\ln(2/\varepsilon)\Bigr),
\end{equation}
where $C>0$ is an absolute constant. Let $A$ be again defined by \eqref{eq:ndefA}. Then
\[
{\mathbb P}\bigl(\delta_k(A)\le \delta\bigr)\ge 1-\varepsilon.
\]
\end{theo}

Two important aspects of compressed sensing are not discussed here at all, namely stability and robustness.
By stability we mean, that the unknown vector $x$ does not have to be exactly sparse, it might have only few large
entries and a long tail of negligible ones. Robustness of the methods corresponds to the fact that the measurements
might be corrupted by some additional noise. Although we do not give any details on that, we just mention that the
results of compressed sensing can be adapted to accomplish both these challenges.

\subsubsection{Briefly about matrices}


If $A\in\R^{m\times N}$ then there is a factorization $A=U\Sigma V^T$, where $U\in\R^{m\times m}$ is an orthogonal matrix,
$\Sigma\in\R^{m\times N}$ is a diagonal matrix with non-negative \emph{singular values} $\sigma_1(A)\ge\sigma_2(A)\ge \dots\ge \sigma_m(A)\ge 0$ on the diagonal,
and $V\in\R^{N\times N}$ is also an orthogonal matrix.

If the matrix $A\in\R^{m\times N}$ has $\rank(A)=r\le m\le N$, we may prefer the so-called ``compact SVD'' $A=U\Sigma V^T$, where
$U\in \R^{m\times r}$ has $r$ mutually orthonormal columns,
$\Sigma\in\R^{r\times r}$ with $\sigma_1(A)\ge \sigma_2(A)\ge \dots\ge \sigma_r(A)>0$ are the non-zero singular values of $A$
and $V\in \R^{N\times r}$ has also $r$ orthonormal columns.
If we denote the columns of $U$ by $u_1,\dots,u_r$ and the columns of $V$ by $v_1,\dots,v_r$, we obtain
\begin{equation}\label{eq:SVDr}
Ax=\sum_{j=1}^r \sigma_j(A) \langle v_j,x\rangle u_j.
\end{equation}

\begin{definition}
Let $A,B\in\R^{m\times N}$. We define the Frobenius (very often called also Hilbert-Schmidt) scalar product of $A$ and $B$ as $\langle A,B\rangle_F:=\sum_{j=1}^m\sum_{k=1}^NA_{j,k}B_{j,k}$.
Similarly, $\|A\|_F:=\sqrt{\langle A,A\rangle_F}$ is called the Frobenius (or Hilbert-Schmidt) norm.
\end{definition}
Let us observe that
\begin{align*}
\tr (A^TB)&=\sum_{k=1}^N (A^TB)_{k,k}=\sum_{k=1}^N\sum_{j=1}^m (A^T)_{k,j}B_{j,k}=\sum_{j=1}^m\sum_{k=1}^NA_{j,k}B_{j,k}\\
&=\langle A,B\rangle_F=\langle B,A\rangle_F=\tr(B^TA).
\end{align*}
Similarly (keyword: \emph{trace is cyclic}) we obtain for $A\in\R^{m\times N}$ and $B\in\R^{N\times m}$ also
$$
\tr(AB)=\sum_{j=1}^m (AB)_{j,j}=\sum_{j=1}^m\sum_{k=1}^N A_{j,k}B_{k,j}=\sum_{k=1}^N\sum_{j=1}^m B_{k,j}A_{j,k}=\sum_{k=1}^N (BA)_{k,k}=\tr(BA).
$$
We then obtain that any two of the expressions $\tr(ABC)$, $\tr(CAB)$ and $\tr(BCA)$ are equal if they are well defined.
\begin{lemma}\label{lem:Pietsch} Let $A\in\R^{m\times N}$. Let $(\varphi_j)_{j=1}^N$ and $(\psi_j)_{j=1}^N$ be two orthonormal basis of $\R^N$. Then
$$
\sum_{j=1}^N\|A\varphi_j\|_2^2=\sum_{j=1}^N\|A\psi_j\|_2^2=\|A\|_F^2.
$$
\end{lemma}
\begin{proof} We decompose $\displaystyle \psi_j=\sum_{k=1}^N\langle \psi_j,\varphi_k\rangle \varphi_k$ and $\displaystyle A\psi_j=\sum_{k=1}^N\langle \psi_j,\varphi_k\rangle A\varphi_k$.
Hence
\begin{align*}
\sum_{j=1}^N\|A\psi_j\|_2^2&=\sum_{j=1}^N \sum_{k=1}^N\sum_{l=1}^N\langle\psi_j,\varphi_k\rangle\langle\psi_j,\varphi_l\rangle \langle A\varphi_k,A\varphi_l\rangle\\
&=\sum_{k=1}^N\sum_{l=1}^N \langle A\varphi_k,A\varphi_l\rangle\sum_{j=1}^N \langle\psi_j,\varphi_k\rangle\langle\psi_j,\varphi_l\rangle 
=\sum_{k=1}^N\sum_{l=1}^N \langle A\varphi_k,A\varphi_l\rangle \langle\varphi_k,\varphi_l\rangle\\
&=\sum_{k=1}^N\|A\varphi_k\|_2^2.
\end{align*}
Choosing the canonical basis $(e_j)_{j=1}^N$ gives the second identity.
\end{proof}
\begin{definition} Let $A\in\R^{n\times N}$. Then we define
\begin{equation}\label{eq:Schatten}
\|A\|_{S_p}:=\begin{cases}\displaystyle \Bigl(\sum_{j=1}^n \sigma_j(A)^p\Bigr)^{1/p}&\quad \text{for}\ 0<p<\infty,\\
\displaystyle \max_{j=1,\dots,n}\sigma_j(A)&\quad \text{for}\ p=\infty.\end{cases}
\end{equation}
\end{definition}
If $p=\infty$, then
$$
\|A\|_{S_\infty}=\sigma_1(A)=\sup_{v\in\R^N,\|v\|_2=1}\|Av\|_2=\sup_{v\in\R^N:\|v\|_2=1}\sup_{u\in\R^n:\|u\|_2=1}\langle u,Av\rangle
$$
is the operator norm and will be denoted by just $\|A\|.$
Indeed, by \eqref{eq:SVDr} we get $\|A\|_{S_\infty}=\langle u_1,Av_1\rangle$ and for any $u\in\R^n$ and $v\in\R^N$ with unit norms we get
by H\"older's inequality
$$
\langle u,Av\rangle = \sum_{j=1}^r \sigma_j(A) \langle v_j,v\rangle \langle u_j,u\rangle\le\sigma_1(A)\Bigl(\sum_{j=1}^r|\langle v_j,v\rangle|^2\Bigr)^{1/2}
\Bigl(\sum_{j=1}^r|\langle u_j,u\rangle|^2\Bigr)^{1/2}\le\sigma_1(A).
$$
By Lemma \ref{lem:Pietsch} we also get $\|A\|_{S_2}=\|A\|_F.$ Indeed, it is enough to take any orthonormal basis of $\R^N$, which includes also the vectors $v_1,\dots,v_r.$

The analogue of the $\ell_1$-norm for matrices is the Schatten-1 norm, also known as \emph{nuclear norm} $\|A\|_*:=\|A\|_{S_1}=\sum_j\sigma_j(A).$
The easiest way to show that this expression is indeed a norm is most likely by showing that the nuclear norm is dual to the operator norm with respect to
the Frobenius scalar product. The reader may want to compare this proof with the proof of the triangle inequality for the $\ell_1$-norm.

\begin{lemma} Let $A\in\R^{n\times N}$. Then
\begin{equation}
\|A\|_*=\sup_{B\in\R^{n\times N}, \|B\|\le 1}\langle A,B\rangle_F.
\end{equation}
\end{lemma}
\begin{proof} ``$\le$'': Let $A=U\Sigma V^T$ and let $B:=U I_n V^T$, where $U,\Sigma,I_n\in\R^{n\times n}$ and $V\in\R^{N\times n}$. Then $\|B\|=1$ and
\begin{align*}
\langle A,B\rangle_F&=\tr(A^TB)=\tr((U\Sigma V^T)^T(UI_nV^T))=\tr(V\Sigma U^TUI_nV^T)=\tr(V\Sigma V^T)\\
&=\tr(V^TV\Sigma)=\tr(I_n\Sigma)=\tr(\Sigma)=\|A\|_*.
\end{align*}
``$\ge$'': If, on the other hand, $A=U\Sigma V^T$ and $\|B\|\le 1$, then we obtain
\begin{align*}
\langle A,B\rangle_F&=\tr(A^TB)=\tr[(U\Sigma V^T)^TB]=\tr(V\Sigma U^TB)=\tr(\Sigma U^TBV)\\
&=\langle\Sigma,U^TBV \rangle_F=\sum_{j=1}^n\sigma_j(A)(U^TBV)_{j,j}=\sum_{j=1}^n\sigma_j(A)(u_j^TBv_j)\\
&\le \sum_{j=1}^n\sigma_j(A)\sigma_1(B)\le \|A\|_*.
\end{align*}
\end{proof}
The subadditivity of the nuclear norm follows easily from this lemma:
\begin{align*}
\|A+B\|_*&=\sup_{C\in\R^{n\times N},\|C\|\le 1}\langle A+B,C\rangle_F=\sup_{C\in\R^{n\times N},\|C\|\le 1}\Bigl(\langle A,C\rangle_F+\langle B,C\rangle_F\Bigr)\\
&\le\sup_{C\in\R^{n\times N},\|C\|\le 1}\langle A,C\rangle_F+\sup_{C\in\R^{n\times N},\|C\|\le 1}\langle B,C\rangle_F=\|A\|_*+\|B\|_*.
\end{align*}

For a real squared symmetric matrix $A=A^T$, we denote by $\lambda_j(A)$ its (real) eigenvalues.
Recall, that their sum is equal to its trace - the sum of the elements on the diagonal.
The following lemma is a certain analogue of a triangle inequality for eigenvalues of symmetric matrices
and singular values of rectangular matrices.

\begin{lemma}\label{lem:Lid} (i) Let $A,B\in\R^{d\times d}$ be two symmetric matrices (i.e. $A=A^T, B=B^T$).
Then
$$
\sum_{j=1}^d |\lambda_j(A)-\lambda_j(B)|\le\sum_{j=1}^d|\lambda_j(A-B)|=\|A-B\|_*.
$$
(ii) Let $A,B\in\R^{n\times N}$. Then
$$
\sum_{j=1}^n|\sigma_j(A)-\sigma_j(B)|\le  \sum_{j=1}^n\sigma_j(A-B).
$$
\end{lemma}
\begin{proof} (i) We use the (Jordan) decomposition of $A-B$ 
into its positive and negative part
$$
A-B=(A-B)^+-(A-B)^-
$$
and obtain
$$
\|A-B\|_*=\tr(A-B)^++\tr(A-B)^-.
$$
We put 
$$
C:=A+(A-B)^-=B+(A-B)^+.
$$
Then $C\geM A$ and $C\geM B$. By Weyl's monotonicity principle\footnote{This can be proved from the minimax characterization
of eigenvalues
$$
\lambda_k(A)=\max_{\substack{M\subset \R^d\\\dim(M)=k}} \min_{\substack{x\in M\\\|x\|_2=1}}\langle x,Ax\rangle
\le \max_{\substack{M\subset \R^d\\\dim(M)=k}} \min_{\substack{x\in M\\\|x\|_2=1}}\langle x,Cx\rangle=\lambda_k(C),
$$ where we have used that
$$
\langle x,Ax\rangle=\langle x,Cx\rangle + \langle x,(C-A)x\rangle\le \langle x,Cx\rangle
$$ if $C-A$ is positive semi-definite.}

$\lambda_j(C)\ge \lambda_j(A)$ and
$\lambda_j(C)\ge\lambda_j(B).$ It follows that
\begin{align*}
\lambda_j(A)-\lambda_j(B)&\le \lambda_j(2C)-\lambda_j(A)-\lambda_j(B) \quad \text{and}\\
\lambda_j(B)-\lambda_j(A)&\le \lambda_j(2C)-\lambda_j(A)-\lambda_j(B), \quad \text{hence}\\
|\lambda_j(A)-\lambda_j(B)|&\le \lambda_j(2C)-\lambda_j(A)-\lambda_j(B).
\end{align*}
Summing up, we get
\begin{align*}
\sum_{j=1}^d |\lambda_j(A)-\lambda_j(B)|&\le \tr(2C)-\tr(A)-\tr(B)\\
&=\tr(A+(A-B)^-)+\tr(B+(A-B)^+)-\tr(A)-\tr(B)\\
&=\|A-B\|_*.
\end{align*}
(ii) Put
$$
\tilde A=\left(\begin{matrix}0&A\\A^T&0\end{matrix}\right)\quad\text{and}\quad \tilde B=\left(\begin{matrix}0&B\\B^T&0\end{matrix}\right).
$$
Then $\tilde A$ and $\tilde B$ are $d\times d$ symmetric matrices with $d=n+N.$ Furthermore, the eigenvalues of $\tilde A$
are\footnote{\dots the eigenvectors being $(u_j^T,v_j^T)^T$ and $(u_j^T,-v_j^T)^T$} $(\pm\sigma_1(A),\dots,\pm\sigma_n(A))$
and similarly for $B$ and $A-B$. Applying (i) gives
\begin{align*}
\sum_{j=1}^{n+N}|\lambda_j(\tilde A)-\lambda_j(\tilde B)|&=\sum_{j=1}^n|\sigma_j(A)-\sigma_j(B)|+\sum_{j=1}^n|-\sigma_j(A)+\sigma_j(B)|\\
&=2\sum_{j=1}^n|\sigma_j(A)-\sigma_j(B)|\le \sum_{j=1}^{n+N}|\lambda_j(\tilde A-\tilde B)|\\
&=\sum_{j=1}^n|\sigma_j(A-B)|+\sum_{j=1}^n|-\sigma_j(A-B)|=2\sum_{j=1}^n\sigma_j(A-B).
\end{align*}
\end{proof}

\subsubsection{Setting of low-rank matrix recovery}

It is very well known (and it is the underlying fact explaining the success of data analysis methods like Principal Component Analysis)
that many matrices appearing in applications are of a low-rank, or at least approximatively low-rank. By that we mean that
their distance (most often measured in the Frobenius norm) to some low-rank matrix is small.
It is therefore desirable to identify low-rank matrices from only a limited amount of information given.
Let us formalize the setting.

Let $A\in \R^{n\times N}$ be a matrix of rank $r\ll \min(n,N).$ The information, which we allow, is only linear.
This means, that we are given an output of a linear \emph{information map} ${\mathcal X}:\R^{n\times N}\to \R^m$,
i.e. the vector $(\X(A)_1,\dots,\X(A)_m)^T$. Finally, we would like to recover (``decode'') $A$ (or at least its good approximation)
from $\X(A)$. Altogether, we would like to have good information maps $\X$ and good decoders $\Delta$ such that
$\Delta(\X(A))$ is close to $A$ for all matrices of a low (prescribed) rank $r.$

The performance of a given coder-decoder pair $(\Delta,\X)$ can be measured by the error between $A$ and $\Delta(\X(A))$, i.e. by
$$
E^r(\Delta,\X)=\sup_{A:\rank(A)\le r, \|A\|_F\le 1}\|A-\Delta(\X(A))\|_F.
$$
The search for the best coder-decoder pair can then be expressed by taking the infimum over all possible $(\Delta,\X)$,
$$
E^r=\inf_{(\Delta,\X)}E^r(\Delta,\X)=\inf_{(\Delta,\X)}\sup_{A:\rank(A)\le r, \|A\|_F\le 1}\|A-\Delta(\X(A))\|_F.
$$
Although there are different versions of these quantities, which incorporate also stability and robustness, we will concentrate only
on the model case when $A$ is indeed exactly low-rank and when the measurements $\X(A)$ are noiseless.

Motivated by the methods of compressed sensing, we will consider only the recovery (=decoder) map
given by nucelar norm minimization, i.e.
\begin{equation}\label{eq:Pstar}
\tag{$P_*$}\argmin_{Z\in\R^{n\times N}}\|Z\|_*\quad \text{s.t.}\quad \X(Z)=\X(A).
\end{equation}
We will therefore concentrate on the construction of a good information map $\X$.

\subsection{Rank-$r$ Null Space Property}
\begin{definition} Let $\X:\R^{n\times N}\to\R^m$ be a linear information map, which associates to every $A\in\R^{n\times N}$
a vector $(\X(A)_1,\dots,\X(A)_m)^T\in\R^m$. Let $1\le r\le n$. We say that $\X$ satisfies the \emph{rank r-NSP} if
\begin{equation}\label{eq:rankNSP}
\sum_{j=1}^r\sigma_j(M)<\sum_{j=r+1}^n\sigma_j(M)\quad\text{for all}\quad M\in\kernel\, \X\setminus\{0\}.
\end{equation}
\end{definition}
\begin{theo} Every matrix $A$ with $\rank(A)\le r$ is a unique solution of $(P_*)$ if, and only if, $\X$ has rank-$r$ NSP.
\end{theo}
\begin{proof}$\bullet\Rightarrow$ Assume first that every matrix $A$ with $\rank(A)\le r$ is the unique solution of $(P_*)$, i.e. of
\begin{equation}\label{eq:NSPPstar}
\argmin_{Z\in\R^{n\times N}}\|Z\|_*\quad \text{s.t.}\quad \X(Z)=\X(A).
\end{equation}
Take any $M\in\kernel \X\setminus\{0\}$ and consider its singular value decomposition $M=U\Sigma V^T$ with
$\sigma_1(M),\dots,\sigma_n(M)$ on the diagonal of $\Sigma.$ Put $M_1=U\Sigma_1V^T$ and $M_2=U\Sigma_2V^T$, where
\begin{align*}
\Sigma_1&={\rm diag}(\sigma_1(M),\dots,\sigma_r(M),0,\dots,0)=\left(\begin{matrix}\sigma_1(M)&&&&&0\\&\sigma_2(M)\\&&\dots\\&&&&\sigma_r(M)\\&0&&&&0\\&&&&&&\dots\\&&&&&&&0\end{matrix}\right),\\
\Sigma_2&={\rm diag}(0,\dots,0,\sigma_{r+1}(M),\dots,\sigma_n(M))=\left(\begin{matrix}0&&&&&0\\&\dots\\&&&\sigma_{r+1}(M)\\&0&&&\sigma_{r+2}(M)\\&&&&&\dots\\&&&&&&\sigma_n(M)\end{matrix}\right)
\end{align*}
Then $M=M_1+M_2$ and $\X(-M_2)=\X(M_1-M)=\X(M_1)$. By assumption, $M_1$ is the unique solution of \eqref{eq:NSPPstar}, hence $\|M_1\|_*<\|M_2\|_*$ and
$\X$ has rank-$r$ NSP.\\
$\bullet\Leftarrow:$ Let
$$
\sum_{j=1}^r\sigma_j(M)<\sum_{j=r+1}^n\sigma_j(M)\quad\text{for all}\quad M\in\kernel\, \X\setminus\{0\}
$$
and let $A\in\R^{n\times N}$ have $\rank(A)\le r$. Let $Z\in\R^{n\times N}$ with $Z\not=A$ and $\X(Z)=\X(A).$ We want to show that $\|A\|_*<\|Z\|_*.$
Put $M:=A-Z$. Then $M\in\kernel\, \X\setminus\{0\}$. Then (using Lemma \ref{lem:Lid})
\begin{align*}
\|Z\|_*&=\|M-A\|_*=\sum_{j=1}^n\sigma_j(M-A)\ge\sum_{j=1}^n|\sigma_j(M)-\sigma_j(A)|\\
&=\sum_{j=1}^r|\sigma_j(M)-\sigma_j(A)|+\sum_{j=r+1}^n\sigma_j(M)\ge \sum_{j=1}^r\sigma_j(A)-\sum_{j=1}^r\sigma_j(M)+\sum_{j=r+1}^n\sigma_j(M)\\
&>\sum_{j=1}^r\sigma_j(A)=\|A\|_*.
\end{align*}
\end{proof}

\subsection{Rank-$r$ Restricted Isometry Property}

As already in the area of compressed sensing, the NSP condition is rather difficult to check.
It is therefore convenient to have another condition, which would imply NSP.
It is not surprising that a certain modification of the RIP will do the job.

\begin{definition} Let $\X:\R^{n\times N}\to\R^m$ be a linear information map. We say that it has $\rank$-$r$ Restricted Isometry Property with the constant $\delta_r>0$ if
$$
(1-\delta_r)\|A\|_F^2\le\|{\mathcal X}(A)\|_2^2\le(1+\delta_r)\|A\|_F^2
$$
for all matrices $A\in\R^{n\times N}$ of rank at most $r$.
\end{definition}
As before, RIP again implies NSP - with nearly the same proof as in compressed sensing. Essentially,
one has to replace Euclidean norms by Frobenius norms and $\ell_1$-norms by nucelar norms.
\begin{theo} If $\delta_{2r}<1/3$, then ${\mathcal X}$ has $\rank$-$r$ NSP. Especially, every $A\in\R^{n\times N}$
with $\rank(A)\le r$ is a unique minimizer of
$$
\argmin_{Z\in\R^{n\times N}}\|Z\|_*\quad \text{s.t.}\quad {\mathcal X}(Z)={\mathcal X}(A).
$$
\end{theo}
\begin{proof} \emph{Step 1:} Let $A,Z\in\R^{n\times N}$ with $\langle A,Z\rangle_F=0$ and $\rank(A)+\rank(Z)\le r$. Then $|\langle \X(A),\X(Z)\rangle|\le \delta_r\|A\|_F\cdot\|Z\|_F.$
Indeed, let first $\|A\|_F=\|Z\|_F=1$. Then
$$
2(1-\delta_r)\le \|\X(A\pm Z)\|_2^2\le 2(1+\delta_r)
$$
and 
\begin{align*}
\langle \X(A),\X(Z)\rangle&=\frac{1}{4}\Bigl(\|\X(A+Z)\|_2^2-\|\X(A-Z)\|_2^2\Bigr)\\
&\le \frac{1}{4}\Bigl(2(1+\delta_r)-2(1-\delta_r)\Bigr)=\delta_r.
\end{align*}
A similar calculation also show that $-\langle \X(A),\X(Z)\rangle\le \delta_r$, giving $|\langle \X(A),\X(Z)\rangle|\le \delta_r.$

The general case then follows by homogeneity - we consider $\tilde A=A/\|A\|_F$ and $\tilde Z=Z/\|Z\|_F$ and apply the result just obtained to $\tilde A$
and $\tilde Z$.

\emph{Step 2:} Let $\delta_{2r}<1/3.$ Let $M\in\kernel\, \X\setminus\{0\}$ and consider its singular value decomposition $M=U\Sigma V^T$, where
$\Sigma=\diag(\sigma_1(M),\sigma_2(M),\dots)$. We put
\begin{align*}
M_0&=U\diag(\sigma_1(M),\sigma_2(M),\dots,\sigma_r(M),0,\dots)V^T,\\
M_1&=U\diag(0,\dots,0,\sigma_{r+1}(M),\dots,\sigma_{2r}(M),0,\dots)V^T,\\
\vdots
\end{align*}
Observe that $\langle M_i,M_j\rangle_F=\langle U\diag(\dots)V^T,U\diag(\dots)V^T\rangle_F=\langle \diag(\dots),\diag(\dots)\rangle_F =0$ for $i\not=j.$
Then $0=\X(M)=\X(M_0+M_1+\dots)$ and
\begin{align*}
\|M_0\|_F^2&\le \frac{1}{1-\delta_r}\|\X(M_0)\|_F^2=\frac{1}{1-\delta_r}\langle \X(M_0),\X(-M_1)+\X(-M_2)+\dots\rangle\\
&\le \frac{1}{1-\delta_r}\sum_{j\ge 1}|\langle \X(M_0),\X(M_j)\rangle|\le \frac{\delta_{2r}}{1-\delta_r}\sum_{j\ge 1}\|M_0\|_F\cdot\|M_j\|_F.
\end{align*}
As $M_0\not=0$, we conclude that
$$
\|M_0\|_F\le \frac{\delta_{2r}}{1-\delta_r}\sum_{j\ge 1}\|M_j\|_F.
$$
We denote $S_1=\{1,2,\dots,r\}$, $S_2=\{r+1,\dots,2r\}$, etc.
The proof is then finished by
\begin{align*}
\sum_{j\ge 1}\|M_j\|_F&=\sum_{j\ge 1}\Bigl(\sum_{l\in S_j}\sigma_l(M)^2\Bigr)^{1/2}\le \sum_{j\ge 1}\Bigl(r\max_{l\in S_j}\sigma_l(M)^2\Bigr)^{1/2}
=\sum_{j\ge 1}\sqrt{r}\max_{l\in S_j}\sigma_l(M)\\
&\le \sum_{j\ge 1}\sqrt{r}\min_{l\in S_{j-1}}\sigma_l(M)\le \sum_{j\ge 1}\sqrt{r}\cdot\frac{\sum_{l\in S_{j-1}}\sigma_l(M)}{r}=\frac{\|M\|_*}{\sqrt{r}}
\end{align*}
and
\begin{align*}
\|M_0\|_*&\le \sqrt{r}\|M_0\|_F\le \sqrt{r}\frac{\delta_{2r}}{1-\delta_r}\frac{\|M\|_*}{\sqrt{r}}=\frac{\delta_{2r}}{1-\delta_r}\|M\|_*<\frac{1}{2}\|M\|_*\\
&=\frac{1}{2}\Bigl(\|M_0\|_*+\|M_1+M_2+\dots\|_*\Bigr),
\end{align*}
hence $\|M_0\|_*< \|M_1+M_2+\dots\|_*$ and $\X$ has rank-$r$ NSP.
\end{proof}

\subsection{Information maps with $\rank$-$r$ RIP}

In this part we describe how to construct information maps with small $m$ and $\rank$-$r$ RIP smaller than, say, 1/3.
It comes as no surprise that the most simple case are information maps generated by i.i.d. Gaussian variables.
The proof follows the same pattern as in compressed sensing - we show a concentration bound for one fixed matrix $A$
and apply $\varepsilon$-net argument to cover the whole set of $\rank$-$r$ matrices.

The first lemma is the most simple $\varepsilon$-net construction in $\R^n.$

\begin{lemma}\label{lem:covS} Let $n\in\N$ and let $\varepsilon>0$. Then there is a subset ${\mathcal N}\subset {\mathbb S}^{n-1}$ with
$|{\mathcal N}|\le (1+2/\varepsilon)^n$ such that for every $x\in{\mathbb S}^{n-1}$ there is a $z\in{\mathcal N}$ with $\|x-z\|_2\le\varepsilon$.
\end{lemma}
\begin{proof}Indeed, let ${\mathcal N}=\{z_1,\dots,z_N\}\subset{\mathbb S}^{n-1}$ be (any) maximal subset of ${\mathbb S}^{n-1}$ with $\|z_j-z_k\|_2\ge \varepsilon$
for $j\not=k$. Then the (open) balls $z_j+\varepsilon/2\cdot B_2^n$ are disjoint and all included in $(1+\varepsilon/2)B_2^n$. Comparing the volumes, we get
$$
N\vol(\varepsilon/2\cdot B_2^n)\le \vol((1+\varepsilon/2)B_2^n)
$$
or, equivalently,
$$
N(\varepsilon/2)^n\vol(B_2^n)\le (1+\varepsilon/2)^n\vol(B_2^n),
$$
which gives the result.
\end{proof}
\begin{rem} With virtually no modifications the same result is true also for the unit ball $B_2^n.$
\end{rem}
Although quite natural, we give an explicit definition of an $\varepsilon$-net.
\begin{definition} We say that ${\mathcal N}\subset X$ is an $\varepsilon$-net of the (quasi-)metric space $(X,\varrho)$
if for every $x\in X$ there is $z\in{\mathcal N}$ with $\|x-z\|<\varepsilon$.
\end{definition}
\begin{lemma}\label{lem:norm_net} Let ${\mathcal N}\subset {\mathbb S}^{n-1}$ be an $\varepsilon$-net of ${\mathbb S}^{n-1}$ for some $0<\varepsilon<1.$
Then
$$
\|A\|=\max_{x\in{\mathbb S}^{n-1}}\|Ax\|_2\le \frac{1}{1-\varepsilon}\max_{z\in{\mathcal N}}\|Az\|_2
$$
for every matrix $A$ with $n$ columns.
\end{lemma}
\begin{proof} Let $x\in{\mathbb S}^{n-1}$. Then there is $z\in{\mathcal N}$ with $\|x-z\|_2\le\varepsilon$
and 
$$
\|Ax\|_2\le \|Az\|_2+\|A(x-z)\|_2\le \max_{z\in{\mathcal N}}\|Az\|_2+\|A\|\cdot \varepsilon.
$$
Taking the supremum over $x\in{\mathbb S}^{n-1}$ finishes the proof.
\end{proof}

We denote by $V_{n,k}$ the \emph{Stiefel manifold} of $k\times n$ orthonormal matrices.
\begin{equation}\label{eq:Stiefel}
V_{n,k}=\{U\in\R^{k\times n}:U U^T=I_k\}=\{U\in\R^{k\times n}: U\ \text{has orthonormal rows}\ u_1,\dots,u_k\}.
\end{equation}

\begin{lemma}\label{lem:covV} To every $\varepsilon>0$, there is a set ${\mathfrak N}\subset V_{n,k}$ with $|{\mathfrak N}|\le (1+2/\varepsilon)^{nk}$, such that to every
$V=(v_1,\dots,v_k)^T\in V_{n,k}$ with rows $v_1^T,v_2^T,\dots,v_k^T$ there is $U=(u_1,\dots,u_k)^T\in{\mathfrak N}$ with $\|V-U\|_{2,\infty}:=\max_{j=1,\dots,k}\|v_j-u_j\|_2\le 2\varepsilon$.
\end{lemma}
\begin{proof} Let $\varepsilon>0$. By Lemma \ref{lem:covS}, we can construct an $\varepsilon$-net ${\mathcal N}\subset {\mathbb S}^{n-1}$
with $|{\mathcal N}|\le (1+2/\varepsilon)^n$ elements. We then consider their tensor product
$$
{\mathcal N}_k=\{U=(u_1,\dots,u_k)^T:u_j\in{\mathcal N}\ \text{for all}\ j=1,\dots,k\}.
$$

This set has at most $(1+2/\varepsilon)^{nk}$ elements but, in general, the rows of any $U\in{\mathcal N}$ are not orthogonal.
By definition, to every $V=(v_1,\dots,v_k)^T\in ({\mathbb S}^{n-1})^k={\mathbb S}^{n-1}\times \dots \times {\mathbb S}^{n-1}$ there is an $k$-tuple
$U=(u_1,\dots,u_k)^T\in{\mathcal N}_k$ with $\|U-V\|_{2,\infty}:=\max_{j=1,\dots,k}\|u_j-v_j\|_2\le\varepsilon.$ But the elements of ${\mathcal N}_k$
do not need to lie in $V_{n,k}$ in general.

We therefore obtain the net ${\mathfrak N}$ as the projection of the points from ${\mathcal N}_{k}$ into $V_{n,k}$ in the following way.
If the distance of $U\in{\mathcal N}_k$ to $V_{n,k}$ in the $\|\cdot\|_{2,\infty}$-norm is larger than $\varepsilon>0$, we leave it out. If it is smaller than that,
we add to ${\mathfrak N}$ (one of) the elements $\tilde U\in V_{n,k}$ with $\|U-\tilde U\|_{2,\infty}=\dist(U,V_{n,k})\le \varepsilon.$

If now $V\in V_{n,k}$, then there is an $U\in{\mathcal N}_k$ with $\|U-V\|_{2,\infty}\le \varepsilon$ and to this $U$, there is a $\tilde U\in{\mathfrak N}$ with
$\|\tilde U-U\|_{2,\infty}\le \varepsilon.$ We get therefore $\|V-\tilde U\|_{2,\infty}\le 2\varepsilon.$
\end{proof}
After these preparations we finally define the Gaussian information maps generated by i.i.d. Gaussian random variables.
\begin{definition} (Gaussian information map). Let ${\mathcal X}(A)=(\langle X_j,A\rangle_F)_{j=1}^m\in\R^m$,
where the matrices $X_j\in\R^{n\times N}$ are (normalized) Gaussian, i.e.
$$
(X_j)_{k,l}\sim \frac{1}{\sqrt{m}}\,{\mathcal N}(0,1)\quad\text{i.i.d.}
$$
\end{definition}

There is a number of ways how to count the ``degrees of freedom'' of a $\rank$-$r$ matrix. This or that way, it is $O(r\max(n,N))$.
It is therefore natural, that the number of measurements $m$ has to be larger than this quantity. Actually, we do not need to pay any(! - up to the
multiplicative constants) price to achieve this bound.

\begin{theo} Let $\X:\R^{n\times N}\to\R^m$ be a Gaussian information map. Then it has $\delta_r\le \delta$ with probability at least $1-\varepsilon$, if
$$
m\ge C_\delta\Big(r(n+N)+\ln(2/\varepsilon)\Big).
$$
\end{theo}
\begin{proof} We first derive a concentration inequality for one fixed $A\in\R^{n\times N}$. Then we construct a net in the set of matrices with rank at most $r$.
Finally, we take a union bound.

\emph{Step 1:} Let $A\in\R^{n\times N}$ with $\|A\|_F=1$ be fixed. We use the 2-stability of Gaussians (cf. Lemma \ref{lem:N})  and calculate
\begin{align*}
\|\X(A)\|_2^2&=\sum_{j=1}^m\langle X_j,A\rangle_F^2=\sum_{j=1}^m\Bigl(\sum_{k=1}^n\sum_{l=1}^N (X_j)_{k,l} A_{k,l}\Bigr)^2\\
&=\frac{1}{m}\sum_{j=1}^m\Bigl(\sum_{k=1}^n\sum_{l=1}^N\omega_{j,k,l}A_{k,l}\Bigr)^2\sim\frac{1}{m}\sum_{j=1}^m\Bigl(\omega_j\|A\|_F\Bigr)^2=\frac{1}{m}\sum_{j=1}^m\omega_j^2,
\end{align*}
hence (by Lemma \ref{Lemma:Con})
\begin{align*}
\P\Bigl(\Bigl|\|\X(A)\|_2^2-1\Bigr|\ge \frac{\delta}{2}\Bigr)&=\P\Bigl(\Bigl|\frac{1}{m}\sum_{j=1}^m\omega_j^2-1\Bigr|\ge\frac{\delta}{2}\Bigr)\\
&=\P\Bigl(\sum_{j=1}^m\omega_j^2\ge (1+\delta/2)m\Bigr)+\P\Bigl(\sum_{j=1}^m\omega_j^2\le(1-\delta/2)m\Bigr)\\
&\le 2e^{-Cm\delta^2}.
\end{align*}
If $\|A\|_F$ is not restricted to be equal to one, we use homogenity and obtain
\begin{align*}
\P\Bigl(\Bigl|\|\X(A)\|_2^2-\|A\|_F^2\Bigr|\ge \frac{\delta}{2}\cdot \|A\|_F^2\Bigr)\le 2e^{-Cm\delta^2}.
\end{align*}

\emph{Step 2:} Next we construct an $\varrho>0$  net (in the Frobenius norm) of the set
$$
D_r=\{A\in\R^{n\times N}:\|A\|_F\le 1,\rank(A)\le r\}
$$
with at most
$$
\Bigl(1+10/\varrho\Bigr)^{r(n+N+1)}
$$
elements.

We apply Lemma \ref{lem:covV} with $\varepsilon=\varrho/5$ to obtain a $2\varrho/5$-net ${\mathfrak N}_1\subset V_{n,r}$
and a $2\varrho/5$-net ${\mathfrak N}_2\subset V_{N,r}$, Finally, we apply \eqref{lem:covS} and the following remark to obtain
a $\varrho/5$-net ${\mathfrak N}_3$ of $B_2^r$. The set
$$
{\mathcal N}=\{\tilde U\tilde \Sigma \tilde V^T:\tilde U\in {\mathfrak N}_1, \tilde \Sigma \in {\mathfrak N}_3, \tilde V\in {\mathfrak N}_2\}
$$
has at most
\begin{align*}
\Bigl(1+\frac{10}{\varrho}\Bigr)^{r}\cdot \Bigl(1+\frac{10}{\varrho}\Bigr)^{nr}\cdot \Bigl(1+\frac{10}{\varrho}\Bigr)^{rN}=
\Bigl(1+\frac{10}{\varrho}\Bigr)^{r(1+n+N)}
\end{align*}
elements.


Let now $A\in D_r$ with singular value decomposition $A=U\Sigma V^T$ and $\tilde A=\tilde U\tilde \Sigma\tilde V^T$.
Here, of course, $\tilde U\in {\mathfrak N}_1$ with $\|U-\tilde U\|_{2,\infty}<2\varrho/5$, 
$\tilde V\in {\mathfrak N}_2$ with $\|V-\tilde V\|_{2,\infty}<2\varrho/5$,
and $\tilde \Sigma\in{\mathfrak N}_3$ with $\|\tilde \Sigma-\Sigma\|_F<\varrho/5.$

Then
\begin{align*}
\|A-\tilde A\|_F&=\|U\Sigma V^T-\tilde U\tilde \Sigma\tilde V^T\|_F\\
&\le \|(U-\tilde U)\Sigma V^T\|_F+\|\tilde U(\Sigma-\tilde\Sigma)V^T\|_F+\|\tilde U\tilde \Sigma (V-\Tilde V)^T\|_F\\
&= \|(U-\tilde U)\Sigma \|_F+\|\Sigma-\tilde\Sigma\|_F+\|\tilde \Sigma (V-\Tilde V)^T\|_F\\
&=\|(\sigma_1(u_1-\tilde u_1),\dots,\sigma_k(u_k-\tilde u_k))\|_F+\|\sigma-\tilde\sigma\|_2\\
&\qquad+\|(\tilde\sigma_1(v_1-\tilde v_1),\dots,\tilde\sigma_k(v_k-\tilde v_k))\|_F\\
&=\Bigl(\sum_{j=1}^k\sigma_j^2\|u_j-\tilde u_j\|_2^2\Bigr)^{1/2}+\|\sigma-\tilde\sigma\|_2+\Bigl(\sum_{j=1}^k\tilde \sigma_j^2\|v_j-\tilde v_j\|_2^2\Bigr)^{1/2}\\
&\le \frac{4}{5}\cdot\varrho+\frac{1}{5}\cdot\varrho=\varrho.
\end{align*}
Therefore, ${\mathfrak N}$ is an $\varrho$-net of $D_r$ in the Frobenius norm.

\emph{Step 3:} By union bound,
\begin{equation}\label{eq:temp1}
\P\Bigl(\Bigl|\|\X(A)\|_2^2-\|A\|_F^2\Bigr|\le \frac{\delta}{2}\cdot\|A\|_F^2\quad\text{for all}\quad A\in{\mathfrak N}\Bigr)\ge 1-2\Bigl(1+\frac{10}{\varrho}\Bigr)^{r(n+N+1)}2^{-Cm\delta^2}
\end{equation}
and, by Lemma \ref{lem:norm_net},
\begin{equation}
\Bigl|\|\X(A)\|_2^2-\|A\|_F^2\Bigr|\le \frac{\delta/2}{1-\varrho}\|A\|_F^2
\end{equation}
for all $A\in D_r$ with at least the same probability. We chose $\varrho=1/2$, which leads to
\begin{equation}
\Bigl|\|\X(A)\|_2^2-\|A\|_F^2\Bigr|\le \delta\|A\|_F^2
\end{equation}
for all $A\in D_r$ with probability at least
$$
1-2\Bigl(1+20\Bigr)^{r(n+N+1)}2^{-Cm\delta^2}.
$$
Putting this larger than $1-\varepsilon$ leads to
$$
2\cdot 21^{r(n+N+1)}2^{-Cm\delta^2}\le \varepsilon
$$
and
$$
\frac{2}{\varepsilon}\cdot 21^{r(n+N+1)}\le 2^{Cm\delta^2},
$$
i.e.
$$
\ln_2(2/\varepsilon)+ r(n+N+1)\ln(21)\le Cm\delta^2
$$
and the proof is finished.

\end{proof}
\newpage

\section{Random Matrices}

The main tool in the analysis of low-rank matrix completion are concentration inequalities
of random matrices. The main aim of this section is to collect the basic results from this area.
Especially, we shall prove an analogue of Lemma \ref{Lemma:Con} on concentration of measure for random matrices.
Before we come to that, let us present the classical proof of the Bernstein inequality for random variables
and let us point out, why this proof can not be directly generalized to the non-commutative case of random matrices.

\begin{lemma}\label{lem:Bernstein1} Let $\omega_1,\dots,\omega_m$ be independent identically distributed random variables with $\E \omega_j=0$, $\E\omega_j^2\le V_0^2$ and $|\omega_j|\le 1$
almost surely for every $j=1,\dots,m.$ Then
$$
\P(|\omega_1+\dots+\omega_m|>t)\le \begin{cases}
\displaystyle 2\exp\Bigl(-\frac{t^2}{4mV_0^2}\Bigr)\qquad &\text{for}\quad t\le 2mV_0^2,\\
\displaystyle 2\exp\Bigl(-\frac{t}{2}\Bigr)\quad &\text{for}\quad t\ge 2mV_0^2.
\end{cases}
$$
\end{lemma}
\begin{proof}
We will estimate only $\P(\omega_1+\dots+\omega_m>t)$, with the second case being symmetric. We get
\begin{align*}
\P(\omega_1+\dots+\omega_m>t)&=\P(\lambda(\omega_1+\dots+\omega_m)>\lambda t)=
\P(\exp(\lambda\omega_1+\dots+\lambda\omega_m-\lambda t)>1)\\
&\le \E\exp(\lambda\omega_1+\dots+\lambda\omega_m-\lambda t)=e^{-\lambda t}[\E\exp(\lambda x_1)]^m
\end{align*}
for every $\lambda > 0.$ If also $1\ge\lambda$, we have $|\lambda \omega_1|\le 1$ almost surely and using
$1+u\le \exp(u)\le 1+u+u^2$ for every $1\ge u\ge -1$, we can further proceed
\begin{align}\label{eq:Bern'}
\P(\omega_1+\dots+\omega_m>t)&\le e^{-\lambda t}[1+\E (\lambda \omega_j)+\E (\lambda \omega_j)^2]^m\\
\notag&=e^{-\lambda t}[1+\E (\lambda x_j)^2]^m\le e^{-\lambda t}[1+\lambda^2V_0^2]^m\le e^{-\lambda t+\lambda^2V_0^2m}.
\end{align}
We now optimize over $0\le\lambda\le 1$ and put $\lambda=\frac{t}{2V_0^2m}$ for $t\le 2V_0^2m$ and $\lambda=1$ for $t\ge 2V_0^2m$,
in which case we get $\P(\dots)\le e^{-t+t/2}=e^{-t/2}.$
\end{proof}
Although this calculation is quite simple, it fails in several aspects when dealing with non-commutative random matrices:
\begin{itemize}
\item The absolute value has to be replaced by another way how to measure the distance between the mean and the actual value of a random matrix.
For matrices, we have several norms to choose from. Although they are mutually equivalent, the constants may depend on the size of the matrix.
\item The most natural candidate for ordering of matrices is the partial ordering $A\leM B$ for $B-A$ being positive semi-definite.
\item Last (but probably the most important) is the failure of the identity $\exp(A+B)=\exp(A)\exp(B)$ for non-commuting matrices $A$ and $B$.
\end{itemize}

\subsection{Golden-Thompson inequality}

This section follows the note \cite{V1} combined with some ideas from the blog of Terry Tao and other similar sources on that topic.


For a matrix $A\in\C^{n\times n}$ we define its matrix exponential function by
$$
\exp(A)=Id+A+\frac{A^2}{2}+\dots=\sum_{j=0}^\infty\frac{A^j}{j!},
$$
where $Id$ is the $n\times n$ identity matrix and $A^j$ is the $j^{\rm th}$ power of $A$. 
This formula can be used to derive several elementary properties of the matrix exponential. For example, it follows that
$$
\|\exp(A)\|\le\exp(\|A\|).
$$
Let us assume that $A$ can be diagonalized (which is the case for example for Hermitian
or real symmetric matrices)
as $A=U\Sigma U^*$, where $\Sigma$ is a diagonal matrix with (complex) $\lambda_1,\dots,\lambda_n$ on the diagonal.
Then $A^j=(U\Sigma U^*)^j=U\Sigma^jU^*$ and
$$
\exp(A)=U\Bigl(\sum_{j=0}^\infty \frac{\Sigma^j}{j!}\Bigr)U^*=U\exp(\Sigma)U^*,
$$
where $\exp(\Sigma)$ is a diagonal matrix with $\exp(\lambda_1),\dots,\exp(\lambda_n)$ on the diagonal.
Observe, that if $A$ is Hermitian (or real symmetric), then its eigenvalues are real and its exponential is therefore positive definite.

Finally, let us recall that if $A,B\in\C^{n\times n}$ are general non-commuting matrices, then the identity $\exp(A+B)=\exp(A)\cdot\exp(B)$
does not need to hold. Nevertheless, it is good to keep in mind, that this identity \emph{holds} if the matrices $A$ and $B$ commute - with
esentially the same proof as for real or complex variables. This is for example the case, when $A=B$ or when $A=Id.$
A suitable replacement is the Golden-Thompson inequality for trace-exponential mapping $A\to \tr(\exp(A))$, see below.

\begin{theo} (Lie Product Formula). For arbitrary matrices $A,B\in\C^{n\times n}$ it holds
\begin{equation}\label{eq:GT:Lie1}
e^{A+B}=\lim_{N\to\infty}(e^{A/N}e^{B/N})^N.
\end{equation}
\end{theo}
\begin{proof}
Let $A,B\in\C^{n\times n}$ be fixed and let us denote
$$
X_N=e^{(A+B)/N}\quad\text{and}\quad Y_N=e^{A/N}e^{B/N}.
$$
By the Taylor's expansion we get
\begin{align*}
X_N&=Id+\frac{A+B}{N}+O(N^{-2}),\\
Y_N&=\left[Id+\frac{A}{N}+O(N^{-2})\right]\cdot\left[Id+\frac{B}{N}+O(N^{-2})\right]\\
&=Id+\frac{A}{N}+\frac{B}{N}+O(N^{-2}).
\end{align*}
This shows that $X_N-Y_N=O(N^{-2}).$ Using the telescopic sum, we obtain
\begin{align*}
e^{A+B}-(e^{A/N}e^{B/N})^N&=X_N^N-Y_N^N\\
&=(X_N^N-X_N^{N-1}Y_N)+(X_N^{N-1}Y_N-X_N^{N-2}Y_N^2)+\dots+(X_NY_N^{N-1}-Y_N^N)\\
&=X_N^{N-1}(X_N-Y_N)+X_N^{N-2}(X_N-Y_N)Y_N+\dots+(X_N-Y_N)Y_N^{N-1}
\end{align*}
and finally
\begin{align*}
\|X_N^N-Y_N^N\|&\le \|X_N-Y_N\|\cdot(\|X_N^{N-1}\|+\|X_N^{N-2}\|\cdot\|Y_N\|+\dots+\|Y_N^{N-1}\|)\\
&\le N\|X_N-Y_N\|\max(\|X_N\|,\|Y_N\|)^{N-1}.
\end{align*}
By Taylor's expansion, we have
\begin{align*}
\|X_N\|&\le e^{\|A+B\|/N}\le e^{\|A\|/N}\cdot e^{\|B\|/N}\quad\text{and}\quad \|Y_N\|\le \|e^{A/N}\|\cdot\|e^{B/N}\|\le e^{\|A\|/N}\cdot e^{\|B\|/N},
\end{align*}
hence
\begin{align*}
\|X_N^N-Y_N^N\|&\le N O(N^{-2})e^{(\|A\|+\|B\|)\cdot (N-1)/N}\to 0 \quad\text{as}\ N\to\infty.
\end{align*}
\end{proof}

Due to the finite dimension of all objects involved, \eqref{eq:GT:Lie1} holds in any norm on $\C^{n\times n}$ as well as
for the convergence in all entries. In particular, we obtain
$$
\tr\bigl(e^{A+B}\bigr)=\lim_{N\to\infty}\tr\Bigl[(e^{A/N}e^{B/N})^N\Bigr].
$$

\begin{theo} (Golden-Thompson inequality). Let $A,B\in\C^{n\times n}.$ Then
\begin{equation}
|\tr(e^{A+B})|\le \tr(e^{(A+A^*)/2}e^{(B+B^*)/2}).
\end{equation}
If $A$ and $B$ are self-adjoint,
\begin{equation}
|\tr(e^{A+B})|\le \tr(e^{A}e^{B}).
\end{equation}
\end{theo}
\begin{proof}
For a natural number $N$, put $X=e^{A/2^N}$ and $Y=e^{B/2^N}$. We will show that
\begin{equation}\label{eq:GT:1}
|\tr[(XY)^{2^N}]|\le \tr[(XX^*)^{2^{N-1}}(YY^*)^{2^{N-1}}],
\end{equation}
i.e.
\begin{equation}\label{eq:GT:2}
|\tr[(e^{A/2^N}e^{B/2^N})^{2^N}]|\le \tr[(e^{A/2^N}e^{A^*/2^N})^{2^{N-1}}(e^{B/2^N}e^{B^*/2^N})^{2^{N-1}}].
\end{equation}

The left-hand side then converges by the Lie Product Formula to $|\tr(e^{A+B})|$ and the right-hand side to $\tr(e^{(A+A^*)/2}e^{(B+B^*)/2})$.

The proof of \eqref{eq:GT:1} is based on the following two simple facts:
\begin{enumerate}
\item[(i)] The trace is cyclic, i.e.
$$
\tr(XY)=\sum_{j=1}^n (XY)_{jj}=\sum_{j=1}^n\sum_{k=1}^n X_{jk}Y_{kj}=\sum_{k=1}^n\sum_{j=1}^nY_{kj}X_{jk}=\sum_{k=1}^n (YX)_{kk}=\tr(YX).
$$
\item[(ii)] For an arbitrary $W,Z\in\C^{n\times n}$ one has by Cauchy-Schwartz inequality
\begin{align}
\notag|\tr(WZ)|&=\biggl|\sum_{j=1}^n (WZ)_{jj}\biggr|=\biggl|\sum_{j=1}^n \sum_{k=1}^n W_{jk} Z_{kj}\biggr|
\le \Bigl(\sum_{j,k=1}^n|W_{jk}|^2\Bigr)^{1/2}\cdot\Bigl(\sum_{j,k=1}^n|Z_{jk}|^2\Bigr)^{1/2}\\
\label{eq:GT:3}&=\Bigl(\sum_{j,k=1}^n W_{jk}\overline{W_{jk}}\Bigr)^{1/2} \cdot \Bigl(\sum_{j,k=1}^n Z_{jk}\overline{Z_{jk}}\Bigr)^{1/2}=\sqrt{ \tr(WW^*) \tr(ZZ^*)}.
\end{align}
\end{enumerate}
This can be further generalized to
\begin{align}\label{eq:GT:5}
|\tr(A_1A_2\dots A_{2^n})|\le \prod_{j=1}^{2^n}\Bigl(\tr[(A_jA_j^*)^{2^{n-1}}]\Bigr)^{1/2^n}.
\end{align}
Indeed, for $n=1$, this is just \eqref{eq:GT:3}. The induction step is then
\begin{align}
\notag|\tr(A_1A_2\dots A_{2^{n+1}})|&=|\tr[(A_1A_2)(A_3A_4)\dots]|\\
\label{eq:GT:4}&\le \Bigl(\tr[(A_1A_2(A_1A_2)^*)^{2^{n-1}}]\Bigr)^{1/2^n}\cdot \dots\\
\notag&= \Bigl(\tr[(A_1A_2A_2^*A_1^*)^{2^{n-1}}]\Bigr)^{1/2^n}\cdot \dots\\
\notag&= \Bigl(\tr[(A_1^*A_1A_2A_2^*)^{2^{n-1}}]\Bigr)^{1/2^n}\cdot \dots,
\end{align}
where the dots represent similar terms for the pairs $(A_3,A_4),$ etc.

Now we use \eqref{eq:GT:5} again for the $2^n$ matrices $(A_1^*A_1), (A_2^*A_2), \dots,(A_1^*A_1), (A_2^*A_2)$
and obtain
\begin{align*}
\tr[(A_1^*A_1A_2A_2^*)^{2^{n-1}}]&=\tr[(A^*_1A_1)(A_2A_2^*)\dots(A^*_1A_1)(A_2A_2^*)]\\
&\le \{\tr[ ((A_1^*A_1)(A_1^*A_1)^*)^{2^{n-1}}]\}^{2^{n-1}/2^n}\cdot \{\tr[ ((A_2^*A_2)(A_2^*A_2)^*)^{2^{n-1}}]\}^{2^{n-1}/2^n}\\
&= \{\tr[ (A_1^*A_1)^{2^{n}}]\}^{1/2}\cdot \{\tr[ (A_2^*A_2)^{2^{n}}]\}^{1/2}\\
&= \{\tr[ (A_1A_1^*)^{2^{n}}]\}^{1/2}\cdot \{\tr[ (A_2A_2^*)^{2^{n}}]\}^{1/2}.
\end{align*}
Inserting this into \eqref{eq:GT:4}, we obtain
$$
|\tr(A_1A_2\dots A_{2^{n+1}})|\le \prod_{j=1}^{2^{n+1}}\Bigl(\tr[ (A_jA_j^*)^{2^n}]\Bigr)^{1/2^{n+1}},
$$
finishing the proof of \eqref{eq:GT:5}. If $A_1=A_2=\dots=A_{2^N}=Z$, \eqref{eq:GT:5} reduces to
\begin{equation}\label{eq:GT:6}
|\tr(Z^{2^N})|\le \tr[(ZZ^*)^{2^{N-1}}].
\end{equation}

To prove \eqref{eq:GT:1}, we apply \eqref{eq:GT:6} to $Z=XY$ and obtain
\begin{align*}
|\tr[(XY)^{2^N}]|&\le \tr\{[(XY)(XY)^*]^{2^{N-1}}\}=\tr[(XYY^*X^*)^{2^{N-1}}]=\tr[(X^*XYY^*)^{2^{N-1}}]\\
&\le \tr\{[(X^*XYY^*)\cdot (X^*XYY^*)^*]^{2^{N-2}}\}=\tr\{[X^*X(YY^*)^2(X^*X)]^{2^{N-2}}\}\\
&=\tr\{[(X^*X)^2(YY^*)^2]^{2^{N-2}}\},
\end{align*}
where we have used the cyclicity of the trace and \eqref{eq:GT:6} with $Z=(X^*XYY^*)$.
Iterating the same procedure, we further obtain by \eqref{eq:GT:6} with $Z=(X^*X)^2(YY^*)^2$
\begin{align*}
\tr[(XY)^{2^N}]|&\le \tr\{[(X^*X)^2(YY^*)^2]^{2^{N-2}}\}\le \tr \{[(X^*X)^2(YY^*)^2((X^*X)^2(YY^*)^2)^*]^{2^{N-3}}\}\\
&=\tr \{[(X^*X)^2(YY^*)^4(X^*X)^2]^{2^{N-3}}\}=\tr \{[(X^*X)^4(YY^*)^4]^{2^{N-3}}\}
\end{align*}
leading to \eqref{eq:GT:1} after further iterations.

\end{proof}

\subsection{Non-commutative Bernstein inequality}

Lemma \ref{Lemma:Con} in the form of Lemma \ref{lem:Bernstein1} can now be generalized to matrices.
This part follows \cite{G} with many forerunners, cf. \cite{AW,T}. In what follows, $\|\cdot\|$ denotes
the spectral norm of a matrix and $\leM$ stands for the partial ordering, i.e. $A\leM B$ if $B-A$ is positive semi-definite.
Finally, we will use the linearity of the trace to obtain for every random matrix $\E\tr(A)=\tr (\E A)$.
Last, but not least, if $A$ and $B$ are two independent random matrices, we have $\E(AB)=(\E A)(\E B).$

\begin{theo}\label{theo:Bern1} Let $X_1,\dots,X_m$ be i.i.d. self-adjoint random matrices with ${\mathbb E}X_i=0$. Assume that $\|{\mathbb E}[X_i^2]\|\le V_0^2$ and
$\|X_i\|\le c$ almost surely. Then
$$
{\mathbb P}\Bigl(\|X_1+\dots+X_m\|>t\Bigr)\le 2n\exp\Bigl(-\frac{t^2}{4mV_0^2}\Bigr),\quad t\le 2mV_0^2/c
$$
and
$$
{\mathbb P}\Bigl(\|X_1+\dots+X_m\|>t\Bigr)\le 2n\exp\Bigl(-\frac{t}{2c}\Bigr),\quad t> 2mV_0^2/c.
$$
\end{theo}
\begin{proof}
First we define a matrix function
$$
\theta(A)=\begin{cases}0,&\text{if}\ A\leM Id,\\ 1 &\text{if}\ A\not\leM Id.\end{cases}
$$
Also, let us note that if $A$ is positive semi-definite (i.e. $A\geM 0$), then $\theta(A)\le\tr(A).$
Let $\lambda>0$ and let $X$ be a random self-adjoint matrix. Then
\begin{align}
\notag \P(X\not\leM t\cdot Id)&=\P(X-t\cdot Id\not\leM 0)=\P(\lambda X-\lambda t\cdot Id \not\leM 0)\\
\label{eq:later1}&=\P(e^{\lambda X-\lambda t\cdot Id}\not\leM Id)=\E\, \theta(e^{\lambda X-\lambda t\cdot Id})\le \E\,\tr(e^{\lambda X-\lambda t\cdot Id})\\
\notag &=\E\,\tr(e^{\lambda X}\cdot e^{-\lambda t\cdot Id})=\E\,\tr(e^{\lambda X}\cdot (e^{-\lambda t} Id))=e^{-\lambda t}\E\,\tr(e^{\lambda X}).
\end{align}
Take now $\sum_{j=1}^m X_j$ for $X$ and obtain by Golden-Thompson inequality
\begin{align*}
\E\Bigl[\tr\exp\Bigl(\lambda \sum_{j=1}^m X_j\Bigr)\Bigr]&=\E\Bigl[\tr\exp\Bigl(\lambda \sum_{j=1}^{m-1} X_j+\lambda X_m\Bigr)\Bigr]
\le \E \Bigl[\tr \Bigl(\exp \Bigl(\lambda\sum_{j=1}^{m-1}X_j\Bigr)\exp(\lambda X_m)\Bigr)\Bigr]\\
&=\tr\Bigl\{\E\exp \Bigl(\lambda\sum_{j=1}^{m-1}X_j\Bigr)\cdot \E \exp(\lambda X_m) \Bigr\}.
\end{align*}
As both the expected values are positive semi-definite matrices, we may use the inequality $\tr (AB)\le \tr(A)\cdot\|B\|$ and obtain
\begin{align*}
\E\Bigl[\tr\exp\Bigl(\lambda \sum_{j=1}^m X_j\Bigr)\Bigr]&\le \tr\Bigl\{\E\exp \Bigl(\lambda\sum_{j=1}^{m-1}X_j\Bigr)  \Bigr\}\cdot\Bigl\|\E \exp(\lambda X_m)\Bigr\|\\
&= \E \Bigl[\tr \exp \Bigl(\lambda\sum_{j=1}^{m-1}X_j\Bigr)  \Bigr]\cdot\Bigl\|\E \exp(\lambda X_m)\Bigr\|.
\end{align*}
This procedure can be iterated until we reach
\begin{align*}
\E\Bigl[\tr\exp\Bigl(\lambda \sum_{j=1}^m X_j\Bigr)\Bigr]&= \E \Bigl[\tr \exp (\lambda X_1)  \Bigr]\cdot\prod_{j=2}^m \Bigl\|\E \exp(\lambda X_j)\Bigr\|.
\end{align*}
Using that $X_j$ are identically distributed and that 
$$
\E \Bigl[\tr \exp (\lambda X_1)  \Bigr] = \tr\, \E \exp (\lambda X_1)\le n\|\E \exp (\lambda X_1)\|,
$$
we get
$$
\P\Bigl(\sum_{j=1}^m X_j\not\leM t\cdot Id\Bigr)\le e^{-\lambda t}\E\Bigl[\tr\exp\Bigl(\lambda \sum_{j=1}^m X_j\Bigr)\Bigr]\le n e^{-\lambda t}\Bigl\|\E \exp(\lambda X_1)\Bigr\|^m.
$$
As $1+y\le e^y\le 1+y+y^2$ for real $-1\le y\le 1$, we get $e^Y\leM Id+Y+Y^2$ for $\|Y\|\le 1$. Hence (for random self-adjoint $Y$ with $\|Y\|\le 1$ and $\E\, Y=0$)
$$
\E[e^Y]\leM \E[Id + Y + Y^2]=\E[Id+Y^2]=Id+\E[Y^2]\leM \exp(\E[Y^2])
$$
and
$$
\|\E[e^Y]\|\le \|\exp(\E[Y^2])\|=e^{\|\E[Y^2]\|}.
$$
If $\|\lambda X_1\|\le 1$ (i.e. if $\lambda\le 1/c$) we may estimate
$$
\Bigl\|\E \exp(\lambda X_1)\Bigr\|^m\le [\exp(\lambda^2V_0^2)]^m
$$
and (cf. \eqref{eq:Bern'})
$$
\P\Bigl(\sum_{j=1}^m X_j\not\leM t\cdot Id\Bigr)\le n\exp(-\lambda t+m\lambda^2V_0^2).
$$
If $t<2mV_0^2/c$, we choose $\lambda=\frac{t}{2mV_0^2}<1/c$ and get $\le n\exp(-t^2/(4mV_0^2))$.
If $t\ge 2mV_0^2/c$, we choose $\lambda=1/c$ and get $\le n\exp(-t/c+mV_0^2/c^2)\le n\exp(-t/c+ct/(2c^2))=n\exp(-t/(2c)).$

Finally, $\|X_1+\dots+X_m\|>t$ if $X_1+\dots+X_m\not\leM t\cdot Id$ or $X_1+\dots+X_m\not \geM -t\cdot Id$, explaining the factor 2.
\end{proof}

\subsection{Lieb's theorem}

An alternative approach to non-commutative Bernstein inequality is based on Lieb's theorem, and is due to Tropp \cite{T}.
We will need the notion of a logarithm of a matrix, but it will be enough to consider positive definite matrices.
If $A$ is a positive definite Hermitian matrix, then it can be written (in a unique way) as $A=\exp(X)$, where $X$
is a Hermitian matrix. This matrix $X$ is then called logarithm of $A$, i.e. $X=\log(A).$

\begin{theo} (Lieb). Fix a self-adjoint matrix $H$. Then the function
$$
A\to \tr\exp(H+\log(A))
$$
is concave on the cone of positive definite Hermitian matrices.
\end{theo}
Lieb's theorem allows for an estimate of the expected value of a trace exponential. Indeed, let $H$ be a fixed self-adjoint matrix
and let $X$ be a random self-adjoint matrix. Define random positive definite matrix $Y=e^X$. Then (by Lieb's theorem and Jensen's inequality)
\begin{align}
\notag \E\,\tr\exp(H+X)&=\E\,\tr\exp(H+\log(e^X))=\E\,\tr\exp(H+\log(Y))\\
\label{eq:Lieb:1}&\le \tr\exp(H+\log(\E\, Y))=\tr\exp(H+\log(\E\, e^X)).
\end{align}
Different choices of $H$ then finally lead to an estimate
\begin{align}
\notag \E\,\tr\exp \Bigl(\sum_{j=1}^N \theta X_j\Bigr)&=\E_1\,\E_2\,\dots\E_N\,\tr\exp \Bigl(\underbrace{\sum_{j=1}^{N-1} \theta X_j}_{H}+\theta X_N\Bigr)\\
\notag &\le\E_1\,\E_2\,\dots\E_{N-1}\,\tr\exp \Bigl(\sum_{j=1}^{N-1} \theta X_j+\log(\E_N\, e^{\theta X_N})\Bigr)\\
\label{eq:Lieb:4} &=\E_1\,\E_2\,\dots\E_{N-1}\,\tr\exp \Bigl(\underbrace{\sum_{j=1}^{N-2} \theta X_j+\log(\E_N\, e^{\theta X_N})}_{H}+\theta X_{N-1}\Bigr)\\
\notag &\le\E_1\,\E_2\,\dots\E_{N-2}\,\tr\exp \Bigl(\sum_{j=1}^{N-2} \theta X_j+\log(\E_N\, e^{\theta X_N})+\log(\E_{N-1}\, e^{\theta X_{N-1}})\Bigr)\\
\notag &\vdots\\
\notag &\le\E_1\,\tr\exp \Bigl(\underbrace{\sum_{j=2}^{N} \log(\E_j\, e^{\theta X_j})}_{H}+\theta X_1\Bigr)\\
\notag &\le \tr\exp\Bigl(\sum_{j=1}^N \log(\E_j\, e^{\theta X_j})\Bigr).
\end{align}

We use this approach to prove the following version of Bernstein's inequality.
\begin{theo}
Let $X_j, i=1,\dots,N$ be independent centered (i.e. $\E\, X_j=0$) self-adjoint random $n\times n$ matrices. Assume that for some numbers $K,\sigma>0$
\begin{equation}\label{eq:Lieb:2}
\|X_i\|\le K\ \text{a.s.}\quad \text{and}\quad \Bigl\|\sum_{j=1}^N\E\,X_j^2\Bigr\|\le\sigma^2.
\end{equation}
Then for every $t\ge 0$ we have
\begin{equation}\label{eq:Lieb:3}
\P\Bigl(\Bigl\|\sum_{j=1}^N X_j\Bigr\|> t\Bigr)\le 2n\exp\Bigl(\frac{-t^2/2}{\sigma^2+Kt/3}\Bigr).
\end{equation}
\end{theo}
\begin{proof}
We concentrate (again) only on the case $K=1$. The general case follows by homogeneity.

\underline{\emph{Step 1.}}
We use the estimate
\begin{equation}
\P\Bigl(\lambda_{\max}\Bigl(\sum_{j=1}^N X_j\Bigr)> t\Bigr)=\P\Bigl(\sum_{j=1}^N X_j\not\leM t\cdot Id\Bigr)
\le e^{-\theta t}\E\,\tr \exp\Bigl(\sum_{j=1}^N\theta X_j\Bigr)
\end{equation}
derived earlier, cf. \eqref{eq:later1}. Then we apply \eqref{eq:Lieb:4} and obtain
\begin{align}
\notag\P\Bigl(\lambda_{\max}\Bigl(\sum_{j=1}^N X_j\Bigr)> t\Bigr)&\le e^{-\theta t}\tr\exp\Bigl(\sum_{j=1}^N \log(\E_j\, e^{\theta X_j})\Bigr)\\
\label{eq:Lieb:6}&\le e^{-\theta t}n\lambda_{\max} \biggl(\exp\Bigl(\sum_{j=1}^N \log(\E_j\, e^{\theta X_j})\Bigr)\biggr)\\
\notag&= e^{-\theta t}n \exp \biggl(\lambda_{\max} \Bigl(\sum_{j=1}^N \log(\E_j\, e^{\theta X_j})\Bigr)\biggr).
\end{align}
\underline{\emph{Step 2.}} Let us assume (w.l.o.g.) that $K=1$. Fix now $\theta>0$ and define a smooth function on the real line
$$
f(x)=\frac{e^{\theta x}-\theta x-1}{x^2}\ \text{for}\ x\not=0\ \text{and}\ f(0)=\frac{\theta^2}{2}.
$$
As $f$ is increasing, we get $f(x)\le f(1)$ for all $x\le 1$ and, therefore, also $f(X_j)\leM f(Id)= f(1)\cdot Id$. We get therefore for every $j=1,\dots,N$
\begin{align*}
e^{\theta X_j}=Id+\theta X_j+(e^{\theta X_j}-Id-\theta X_j)=Id + \theta X_j+X_j f(X_j) X_j\leM Id +\theta X_j + f(1)X_j^2
\end{align*}
and, in expectation,
\begin{align*}
\E\, e^{\theta X_j}\leM Id + f(1)\E\,X_j^2\leM \exp\Bigl(f(1)\cdot \E\,X_j^2\Bigr)=\exp((e^\theta-\theta-1)\cdot \E\,X_j^2).
\end{align*}
Plugging this into \eqref{eq:Lieb:6}, we further obtain
\begin{align*}
\P\Bigl(\lambda_{\max}\Bigl(\sum_{j=1}^N X_j\Bigr)> t\Bigr)&\le e^{-\theta t}n \exp \biggl(\lambda_{\max} \Bigl(\sum_{j=1}^N (e^\theta-\theta-1) \E\, X^2_j\Bigr)\biggr)\\
&=e^{-\theta t}n \exp \biggl((e^\theta-\theta-1)\lambda_{\max} \Bigl(\sum_{j=1}^N  \E\, X^2_j\Bigr)\biggr)\\
&\le e^{-\theta t}n \exp \bigl((e^\theta-\theta-1)\sigma^2\bigr).
\end{align*}
Finally, we plug in the minimizer over $\theta>0$, namely $\theta:=\log(1+t/\sigma^2)$ and obtain
\begin{align*}
&\le n\exp \Bigl(-\log(1+t/\sigma^2)t+\Bigl((1+t/\sigma^2)-\log(1+t/\sigma^2)-1\Bigr)\sigma^2\Bigr)\\
&=n\exp \Bigl(-\log(1+t/\sigma^2)(t+\sigma^2)+\frac{t}{\sigma^2}\cdot\sigma^2\Bigr)\\
&=n\exp \Bigl(-\Bigl[\log(1+t/\sigma^2)(1+t/\sigma^2)-t/\sigma^2\Bigr]\sigma^2\Bigr)\\
&=n\exp \Bigl(-\Bigl[h(t/\sigma^2)\Bigr]\sigma^2\Bigr)\le n\exp \Bigl(-\frac{(t/\sigma^2)^2/2}{1+t/(3\sigma^2)}\sigma^2\Bigr)\\
&=n\exp\Bigl(-\frac{t^2/2}{\sigma^2+t/3}\Bigr),
\end{align*}
where elementary calculus shows that 
$$
h(u)=(1+u)\log(1+u)-u\ge \frac{u^2/2}{1+u/3}
$$
for $u\ge 0.$
\end{proof}

\newpage
\section{Low-rank matrix recovery and matrix completion}
This section follows closely \cite{G}.

\subsection{Setting and main results}

We return to the question of recovering a low-rank matrix $A$ from a limited number
of linear measurements ${\mathcal X}(A)=\{\langle X_1,A\rangle_F,\dots,\langle X_m,A\rangle_F\}$.
In contrast to Section \ref{sec:2}, we will now restrict the possible choices of $X_j$'s.

As an important example we keep in mind is that $X_j$'s are chosen from the set $\{e_k e^T_l\}_{k,l}$.
Here, $e_ke_l^T$ is a matrix with the only non-zero entry on $k^{\rm th}$ row and $l^{\rm th}$ column.
Then $\langle e_ke_l^T,A\rangle_F=A_{k,l}$ is one entry of $A$ and we would like to recover a low-rank matrix
by observing only few of its entries.

We will (for simplicity) deal with squared $n\times n$ matrices $A$. We will assume, that $A$ is Hermitian
(or symmetric in the case of real matrices). Its rank will be denoted by $r\in\{1,\dots,n\}$ with $r\ll n$
being of the largest importance. Furthermore, we will assume that $\{X_a\}_{a=1}^{n^2}$ is an orthonormal
basis\footnote{The matrices $X_a$'s do not need to be necessarily self-adjoint.}
of the set of $n\times n$ matrices with respect to the Frobenius (=Hilbert-Schmidt) inner product.
Then
\begin{equation}\label{eq:G:1}
A=\sum_{a=1}^{n^2}\langle X_a,A\rangle_FX_a.
\end{equation}
The most natural setting is then as follows. We observe several randomly chosen scalar products
$$
\langle X_a,A\rangle_F,\quad a\in\Omega,
$$
where $\Omega\subset\{1,\dots,n^2\}$ is chosen at random among all subsets of $\{1,\dots,n^2\}$ with $m$ elements.
Finally, we would like to know when the minimizer of
\begin{equation}\label{eq:G:2}
\argmin_{Z\in\R^{n\times n}} \|Z\|_*,\quad\text{s.t.}\quad \langle X_a,Z\rangle_F=\langle X_a,A\rangle_F\quad\text{for all}\ a\in\Omega
\end{equation}
is unique and equal to $A$ itself.
The random choice of the set $\Omega$ is statistically rather difficult process. We will instead rather assume, that
we are given $m$ independent random variables $\omega_1,\dots,\omega_m$ taking uniformly distributed values
in $\{1,\dots,n^2\}$ and we consider instead of \eqref{eq:G:2} its analogue
\begin{equation}\label{eq:G:3}
\argmin_{Z\in\R^{n\times n}} \|Z\|_*,\quad\text{s.t.}\quad \langle X_{\omega_j},Z\rangle_F=\langle X_{\omega_j},A\rangle_F\quad\text{for all}\ j=1,\dots,m.
\end{equation}
The independence of $\omega$'s makes this approach much easier to analyze. There is nevertheless the danger of ``collisions'', i.e.
it might happen that $\omega_j=\omega_k$ for $j\not=k.$

We can see already now, how random matrices and operators come into play. The matrix $X_{\omega}$ is a random matrix
taking randomly and uniformly distributed the values in $\{X_1,\dots,X_{n^2}\}$. Moreover, we denote by
$$
{\mathcal P}_{\omega}:Z\to \langle X_\omega,Z\rangle_F X_\omega
$$
the projection onto $X_\omega.$ These are random matrix operators, which we combine together into the
sampling operator
\begin{equation}\label{eq:G:4}
{\mathcal R}:Z\to \frac{n^2}{m}\sum_{j=1}^{m}\langle X_{\omega_j},Z\rangle_F X_{\omega_j}.
\end{equation}
This allows us to reformulate \eqref{eq:G:3} once more. We analyze
\begin{equation}\label{eq:G:5}
\argmin_{Z\in\R^{n\times n}} \|Z\|_*\quad\text{s.t.}\quad {\mathcal R}(Z)={\mathcal R}(A).
\end{equation}

Before we come to the main result and its proof, let us make one simple observation.
If (say) $X_1$ is itself of low-rank and $A=X_1$, it will be surely difficult to find by \eqref{eq:G:5}.
Indeed, $\langle A,X_j\rangle_F=0$ for all $j>1$ due to the orthonormality of the basis $\{X_a\}_a$.
If we observe some of the coefficients $\{\langle A,X_a\rangle_F\}$, we might be lucky (if $\langle A,X_1\rangle_F$ is included in the selection)
or unlucky (if this coefficient is not included). The chance of this luck grows with the portion of coefficients observed and a large portion (nearly all)
of them has to observed if the chance of hiting it is supposed to be high. In general, in such a case we can not hope for recovery of $A$ from small number
of its coefficients with respect to this orthonormal basis.

We put $U=\range (A)=[\kernel(A)]^{\perp}$ be the column (and due to the symmetry also the row) space of $A$.
By $P_U$ we denote the orthogonal projection onto $U$. Hence
$$
A=P_UA=AP_U\quad\text{and}\quad P_{U^\perp}A=AP_{U^\perp}=0.
$$
We now express $A$ in an eigenvector basis of $A$. Let $\{u_1,\dots,u_r\}$ be an orthonormal basis of $U$
of eigenvectors of $A$ (i.e. $Au_j=\lambda_ju_j$) and let $u_{r+1},\dots,u_n$ be an orthonormal basis of $U^\perp$.
We can write $A$ with respect to this basis as
$$
A=\left(\begin{tabular}{llll|lll}$\lambda_1$&\dots&0&&&\vdots\\
0&$\lambda_2$&&$\ddots$&&0\\
&&$\ddots$&&&\vdots\\
0&$\ddots$&&$\lambda_r$\\
\hline
&&&&0\\
0&\dots&\dots&0&&$\ddots$\\
&&&&&&0
\end{tabular}\right).
$$
Furthermore, for each $n\times n$ matrix $Z$ we can use the decomposition
$$
Z=(P_U+P_{U^\perp})Z(P_U+P_{U^\perp})
$$
and write $Z$ with respect to the basis $\{u_j\}_{j=1}^n$ in the block form
$$
Z=\left(\begin{tabular}{l|l}$P_UZP_U$&$P_UZP_{U^\perp}$\\\hline$P_{U^\perp}ZP_U$&$P_{U^\perp}ZP_{U^\perp}$
\end{tabular}\right).
$$
Finally, we let $T$ be the matrices which vanish in this notation in the last block and ${\mathcal P}_T$
be the projection onto this subspace, i.e.
\begin{align*}
T&:=\{Z:P_{U^\perp}ZP_{U^\perp}=0\}\quad \text{and}\\
{\mathcal P}_T&:Z\to P_UZP_U+P_UZP_{U^\perp}+P_{U^\perp}ZP_U\\&\qquad\quad=P_UZ+ZP_U-P_UZP_U.
\end{align*}
By the observation above, some additional condition is necessary to guarantee the success of low-rank matrix recovery.
\begin{definition} The $n\times n$ rank-$r$ matrix $A$ has coherence $\nu > 0$ with respect to the operator basis $\{X_a\}_{a=1}^{n^2}$ if either
\begin{equation}\label{eq:G:20}
\max_{a}\|X_a\|^2\le \frac{\nu}{n},
\end{equation}
or
\begin{equation}\label{eq:G:21}
\max_{a}\|{\mathcal P}_TX_a\|_F^2\le \frac{2\nu r}{n}\quad \text{and}\quad\max_{a}\langle X_a,\sgn(A)\rangle_F^2\le \frac{\nu r}{n^2}.
\end{equation}
\end{definition}
The condition \eqref{eq:G:20} is more restrictive - it does not depend on the (unknown) matrix $A$ or its rank.
In other words, matrix bases with \eqref{eq:G:20} has small coherence with respect to all low-rank matrices.
Let us show that \eqref{eq:G:20} indeed implies the first half of \eqref{eq:G:21}. Observing that matrices from $T$ have the rank $2r$ at most we obtain
\begin{align*}
\|{\mathcal P}_TX_a\|_F^2&=\sup_{Z\in T,\|Z\|_F=1}\langle X_a,Z\rangle_F^2\le \sup_{Z\in T,\|Z\|_F=1}\|X_a\|^2\cdot\|Z\|_*^2\\
&\le \sup_{Z\in T,\|Z\|_F=1}2r\|X_a\|^2\cdot\|Z\|_F^2\le \frac{2\nu r}{n}.
\end{align*}
Furthermore, if $\|X_a\|_F=1$, then $\|X_a\|\ge \frac{1}{\sqrt{n}}$ and we see that $\nu$ is always greter or equal to 1 in \eqref{eq:G:20}.

The most important example is surely the operator basis given by $\{e_ie_j^T\}_{i,j=1}^{n^2}.$ Let $U=\range(A)$ and let $A$ satisfy
$$
\max_{i}\|P_Ue_j\|_2^2\le\mu_1\frac{r}{n}\quad\text{and}\quad\max_{i,j}|\langle e_i,\sgn(A)e_j\rangle|\le\mu_2\frac{\sqrt{r}}{n}.
$$
Then we obtain for each $(i,j)\in\{1,\dots,n\}^2$
\begin{align*}
\|{\mathcal P}_T (e_ie_j^T)\|_F^2&=\|P_U(e_ie_j^T)+P_{U^\perp}(e_ie_j^T)P_U\|_F^2\\
&=\|P_U(e_ie_j^T)\|_F^2+ \|P_{U^\perp}(e_ie_j^T)P_U\|_F^2\\
&=\|(P_Ue_i)e_j^T\|_F^2+\|(P_{U^\perp}e_i)(e_j^TP_U)\|_F^2\\
&=\|P_Ue_i\|_2^2\cdot \|e_j^T\|_2^2+\|P_{U^\perp}e_i\|_2^2\cdot \|e_j^TP_U\|_2^2\\
&\le \mu_1\frac{r}{n}\cdot 1+1\cdot \|P_{U}e_j\|_2^2\le 2\mu_1\frac{r}{n}\\
\langle e_ie_j^T,\sgn(A)\rangle_F^2&=\langle e_i,\sgn(A)e_j\rangle_F^2\le \mu_2^2\frac{r}{n^2}.
\end{align*}
Hence, we obtain \eqref{eq:G:20} with $\nu=\max\{\mu_1,\mu_2^2\}$.

The following theorem is the main result of this section. It shows, that with random choice of coefficients
of $A$ w.r.t. the operator basis we indeed recover $A$ with high probability.

\begin{theo} Let $A$ be a $n\times n$ rank-$r$ matrix with coherence $\nu$ with respect to
an operator basis $\{X_a\}_{a=1}^{n^2}$. Let $\Omega\subset\{1,\dots,n^2\}$ be a random set of size
$|\Omega|\ge O(nr\nu(1+\beta)\ln^2n)$. Then the solution of \eqref{eq:G:2} is unique and is equal to $Z$ with probability at least
$1-n^{-\beta}.$
\end{theo}
\begin{proof}\ \\
Let $Z\in\R^{n\times n}$. We put $\Delta=Z-A$. We have to show that $\|Z\|_*=\|\Delta + A\|_*>\|A\|_*$ if ${\mathcal R}(\Delta)=0$
and $\Delta\not=0.$ If ${\mathcal R}(\Delta)\not=0$, then $Z$ is not one of the matrices considered in \eqref{eq:G:5} and we call it infeasible.
Furthermore, we decompose $\Delta=\Delta_T+\Delta_{T^\perp}$, where $\Delta_T={\mathcal P}_T \Delta.$\vskip.3cm
\noindent\underline{Step 1.: Reduction to sampling with collisions} \vskip.3cm

Sampling of a random subset $\Omega\subset\{1,\dots,n^2\}$ with $m$ elements corresponds to sampling of $\omega_1,\dots,\omega_m$ without collisions.
By that we mean, that $\omega_1$ is chosen randomly and uniformly from $\{1,\dots,n^2\}$. Then $\omega_2$ is chosen from $\{1,\dots,n^2\}\setminus\{\omega_1\}$, etc.
We denote the probability of \eqref{eq:G:5} recovering $A$ by $p_{\rm wo}(m)$ when sampling without collisions and by $p_{\rm wi}(m)$ when collisions are allowed.

We define ${\mathcal R}'$ as in \eqref{eq:G:4} but with the sum going only over distinct $\omega_i\not=\omega_j$. The number of distinct
samples will be denoted by $m'\le m$. Then $\kernel\,{\mathcal R}=\kernel\,{\mathcal R}'$
and ${\mathcal R}(Z-A)=0$ if and only if ${\mathcal R}'(Z-A)=0$. Conditioned on $m'$, the distribution of ${\mathcal R}'$ is the same as sampling
$m'$ times without replacement. Hence
$$
p_{\rm wi}(m)=\E_{m'}[p_{\rm wo}(m')]\le p_{\rm wo}(m).
$$
Hence, sampling with replacement is more likely to fail than sampling without and it is enough to show that the probability
of failure when sampling with replacement is tiny.

\noindent\underline{Step 2.: $\Delta_T$ large} \vskip.3cm

Let ${\mathcal R}$ be defined by \eqref{eq:G:4}, i.e.
$$
{\mathcal R}=\frac{n^2}{m}\sum_{j=1}^m {\mathcal P}_{\omega_j}.
$$
The operator norm of ${\mathcal R}:(\R^{n\times n},\|\cdot\|_F)\to(\R^{n\times n},\|\cdot\|_F)$ is equal to $\frac{n^2}{m}$ times the highest number of collisions in one direction. A very rough estimate
is therefore $\|{\mathcal R}\|\le n^2.$ Furthermore,
\begin{equation}\label{eq:G:6}
\E {\mathcal R}=\frac{n^2}{m}\sum_{j=1}^m \E{\mathcal P}_{\omega_j}=n^2\E{\mathcal P}_{\omega}=n^2\cdot \frac{1}{n^2}\sum_{a=1}^{n^2}{\mathcal P}_{a}=Id
\end{equation}
and
\begin{equation*}
\E [{\mathcal P}_T{\mathcal R}{\mathcal P}_T]={\mathcal P}_T[\E {\mathcal R}]{\mathcal P}_T={\mathcal P}_T.
\end{equation*}
We will prove later (using the concentration bounds on matrices) that even more is true, namely that
\begin{equation}\label{eq:G:7}
\|{\mathcal P}_T-{\mathcal P}_T{\mathcal R}{\mathcal P}_T\|<1/2
\end{equation}
with high probability (and let us denote the failure of this event by $p_1$).

Let
\begin{equation*}
\|\Delta_T\|_F^2> 2mn^2\|\Delta_{T^\perp}\|_F^2.
\end{equation*}
Then we obtain
\begin{align*}
\|{\mathcal R}\Delta_{T^\perp}\|^2_F&\le \|{\mathcal R}\|^2\cdot \|\Delta_{T^\perp}\|^2_F\le n^4 \|\Delta_{T^\perp}\|^2_F\\
&< \frac{n^2}{2m}\|\Delta_T\|_F^2\le \frac{n^2}{m}(1-\|{\mathcal P}_T-{\mathcal P}_T{\mathcal R}{\mathcal P}_T\|)\|\Delta_T\|_F^2\\
&\le\frac{n^2}{m}\Bigl\{\langle \Delta_T ,\Delta_T\rangle_F-\langle[{\mathcal P}_T-{\mathcal P}_T{\mathcal R}{\mathcal P}_T]\Delta_T,\Delta_T \rangle_F\Bigr\}\\
&=\frac{n^2}{m}\Bigl\{\langle \Delta_T ,\Delta_T\rangle_F-\langle{\mathcal P}_T\Delta_T,\Delta_T \rangle_F+\langle{\mathcal P}_T{\mathcal R}{\mathcal P}_T\Delta_T,\Delta_T \rangle_F\Bigr\}\\
&= \frac{n^2}{m}\langle \Delta_T, {\mathcal P}_T{\mathcal R}{\mathcal P}_T\Delta_T\rangle_F=\frac{n^2}{m}\langle \Delta_T,{\mathcal R}\Delta_T\rangle_F\\
&=\frac{n^4}{m^2}\sum_{j=1}^{m}\langle \Delta_T,{\mathcal P}_{\omega_j}(\Delta_T)\rangle_F
=\frac{n^4}{m^2}\sum_{j=1}^{m}\langle {\mathcal P}_{\omega_j}(\Delta_T),{\mathcal P}_{\omega_j}(\Delta_T)\rangle_F\\
&\le \frac{n^4}{m^2}\sum_{j,k=1}^{m}\langle {\mathcal P}_{\omega_j}(\Delta_T),{\mathcal P}_{\omega_k}(\Delta_T)\rangle_F
=\langle {\mathcal R}\Delta_T, {\mathcal R}\Delta_T\rangle_F=\|{\mathcal R}\Delta_T\|_F^2
\end{align*}
and we conclude that ${\mathcal R}(\Delta)\not =0$ and $\Delta$ is infeasible.

It remains to prove \eqref{eq:G:7}. We will actually prove that
\begin{equation}\label{eq:G'}
\P(\|{\mathcal P}_T-{\mathcal P}_T{\mathcal R}{\mathcal P}_T\|\ge t)\le 4nr\exp\Bigl(-\frac{t^2 m}{4(2\nu rn+1)}\Bigr),\quad 0<t<2.
\end{equation}
We apply the operator bound Theorem \ref{theo:Bern1} with operators
$$
S_{\omega_j}=\frac{n^2}{m}{\mathcal P}_T{\mathcal P}_{\omega_j}{\mathcal P}_T-\frac{1}{m}{\mathcal P}_T.
$$
We have to verify the setting of this theorem. Therefore we observe couple of facts.
\begin{itemize}
\item $S_{\omega_j}$ are centered:
$$
\E\,S_{\omega_j}=\frac{n^2}{m}{\mathcal P}_T[\E\,{\mathcal P}_{\omega_j}]{\mathcal P}_T-\frac{1}{m}{\mathcal P}_T
=\frac{n^2}{m}{\mathcal P}_T\Bigl[\frac{1}{n^2}\sum_{a=1}^{n^2}{\mathcal P}_a\Bigr]{\mathcal P}_T-\frac{1}{m}{\mathcal P}_T=0.
$$
\item Their sum is the operator to bound:
$$
\sum_{j=1}^m S_{\omega_j}=\frac{n^2}{m}{\mathcal P}_T\Bigl[\sum_{j=1}^m{\mathcal P}_{\omega_j}\Bigr]{\mathcal P}_T-{\mathcal P}_T
={\mathcal P}_T{\mathcal R}{\mathcal P}_T-{\mathcal P}_T.
$$
\item We estimate the value $c$ by \eqref{eq:G:21}
\begin{align*}
\|S_{\omega_j}\|&=\frac{1}{m}\Bigl\|n^2{\mathcal P}_T{\mathcal P}_{\omega_j}{\mathcal P}_T-{\mathcal P}_T \Bigr\|
\le \frac{n^2}{m}\|{\mathcal P}_T{\mathcal P}_{\omega_j}{\mathcal P}_T\| + \frac{1}{m}\|{\mathcal P}_T\|\\
&\le \frac{n^2}{m}\max_a\|{\mathcal P}_T X_a\|_F^2+\frac{1}{m}\le \frac{2\nu r n+1}{m}=:c,
\end{align*}
where we used that
\begin{align*}
\|{\mathcal P}_T{\mathcal P}_{\omega_j}{\mathcal P}_T(Z)\|_F&=\Bigl\|{\mathcal P}_T\Bigl(\langle {\mathcal P}_T(Z),X_{\omega_j} \rangle_FX_{\omega_j}\Bigr)\Bigr\|_F
=\Bigl\|\langle {\mathcal P}_T(Z),X_{\omega_j} \rangle_F {\mathcal P}_T(X_{\omega_j})\Bigr\|_F\\
&=|\langle {\mathcal P}_T(Z),X_{\omega_j} \rangle_F|\cdot \|{\mathcal P}_T(X_{\omega_j})\|_F\\
&=|\langle Z,{\mathcal P}_T(X_{\omega_j}) \rangle_F|\cdot \|{\mathcal P}_T(X_{\omega_j})\|_F\le \|Z\|_F\cdot \|{\mathcal P}_T(X_{\omega_j})\|^2_F.
\end{align*}
\item \dots and $V_0^2$:
\begin{align*}
\|\E[S_{\omega_j}^2]\|&=\Bigr\|\E\Bigl[\Bigr(\frac{n^2}{m}{\mathcal P}_T{\mathcal P}_{\omega_j}{\mathcal P}_T-\frac{1}{m}{\mathcal P}_T\Bigr)^2\Bigr]\Bigl\|\\
&=\Bigl\|\E\Bigl[\Bigl(\frac{n^2}{m}{\mathcal P}_T{\mathcal P}_{\omega_j}{\mathcal P}_T\Bigr)^2-2\frac{n^2}{m^2}{\mathcal P}_T{\mathcal P}_{\omega_j}{\mathcal P}_T+\frac{1}{m^2}{\mathcal P}_T\Bigr]\Bigr\|\\
&=\Bigl\|\E\Bigl(\frac{n^2}{m}{\mathcal P}_T{\mathcal P}_{\omega_j}{\mathcal P}_T\Bigr)^2-\frac{1}{m^2}{\mathcal P}_T\Bigr\|\\
&\le\frac{n^4}{m^2}\|\E[{\mathcal P}_T{\mathcal P}_{\omega_j}{\mathcal P}_T{\mathcal P}_{\omega_j}{\mathcal P}_T]\| +\frac{1}{m^2}.
\end{align*}
As ${\mathcal P}_{\omega_j}{\mathcal P}_T(Z)\in\linspan\{X_{\omega_j}\}$ and
$$
{\mathcal P}_{\omega_j}{\mathcal P}_T(X_{\omega_j})=\langle{\mathcal P}_TX_{\omega_j},X_{\omega_j} \rangle_F\cdot X_{\omega_j},
$$
i.e. on $\linspan(X_{\omega_j})$ the operator ${\mathcal P}_{\omega_j}{\mathcal P}_T$ acts as $\langle{\mathcal P}_TX_{\omega_j},X_{\omega_j} \rangle_F$
times the identity. We use \eqref{eq:G:21} and the fact that  ${\mathcal P}_T{\mathcal P}_{\omega_j}{\mathcal P}_T$ are positive semi-definite, we get
\begin{align*}
\|\E[S_{\omega_j}^2]\|&\le\frac{n^4}{m^2}\|\E[\langle{\mathcal P}_TX_{\omega_j},X_{\omega_j} \rangle_F{\mathcal P}_T{\mathcal P}_{\omega_j}{\mathcal P}_T]\| +\frac{1}{m^2}\\
&\le \frac{n^4}{m^2}\cdot \frac{2\nu r}{n}\|\E[{\mathcal P}_T{\mathcal P}_{\omega_j}{\mathcal P}_T]\| +\frac{1}{m^2}\\
&= \frac{n^4}{m^2}\cdot \frac{2\nu r}{n}\cdot\frac{1}{n^2}\|{\mathcal P}_T\| +\frac{1}{m^2}=\frac{2\nu r n+1}{m^2}=:V_0^2.
\end{align*}
\end{itemize}
By Theorem \ref{theo:Bern1}, we get for $0<t<2mV_0^2/c=2m \cdot\frac{2\nu r n+1}{m^2}\cdot\frac{m}{2\nu r n+1}=2$ the desired inequality with
$$
p_1=4nr\exp\Bigl(-\frac{m}{16(2\nu rn+1)}\Bigr).
$$

Finally, we note that the operators involved can be understood as defined on $T$ only, which has dimension $2rn-r^2\le 2rn.$


\vskip.3cm

\noindent\underline{Step 3.: $\Delta_T$ small}\vskip.3cm

We assume that
\begin{equation}\label{eq:G:10}
\|\Delta_T\|_F< n^2\|\Delta_{T^\perp}\|_F
\end{equation}
and 
\begin{equation}\label{eq:G:11}
{\mathcal R}(\Delta)=0\quad\text{i.e.}\quad \Delta\in (\range {\mathcal R})^\perp.
\end{equation}
We will show that (under additional conditions) this implies that
$$
\|Z\|_*=\|A+\Delta\|_*>\|A\|_*.
$$
Let us recall that $U=\range(A).$

We calculate\footnote{The first inequality $\|Z\|_*\ge \|P_UZP_U\|_*+\|P_{U^\perp}ZP_{U^\perp}\|_*$ is sometimes called pinching inequality.
It can be proved by duality:
\begin{align*}
\|P_UZP_U\|_*+\|P_{U^\perp}ZP_{U^\perp}\|_*&=\sup_{\|A\|\le 1}\langle P_UZP_U,A\rangle_F +\sup_{\|B\|\le 1}\langle P_{U^\perp}ZP_{U^\perp},B\rangle_F\\
&=\sup_{\|A\|\le 1}\langle Z,P_UAP_U\rangle_F +\sup_{\|B\|\le 1}\langle Z,P_{U^\perp}BP_{U^\perp}\rangle_F\\
&=\sup_{\|A\|\le 1,\|B\|\le 1}\langle Z,P_UAP_U+P_{U^\perp}BP_{U^\perp}\rangle_F\\
&\le\sup_{\|C\|\le 1} \langle Z,C\rangle_F=\|Z\|_*.
\end{align*}
}
\begin{align*}
\|A+\Delta\|_*&\ge \|P_U(A+\Delta)P_U\|_*+\|P_{U^\perp}(A+\Delta)P_{U^\perp}\|_*\\
&=\|A+P_U\Delta P_U\|_*+\|\Delta_{T^\perp}\|_*\\
&\ge \langle \sgn(A),A+P_U\Delta P_U\rangle_F+\langle\sgn(\Delta_{T^\perp}),\Delta_{T^\perp} \rangle_F\\
&=\|A\|_*+\langle \sgn(A),P_U\Delta P_U\rangle_F + \langle\sgn(\Delta_{T^\perp}),\Delta_{T^\perp} \rangle_F\\
&=\|A\|_*+\langle \sgn(A)+\sgn(\Delta_{T^\perp}),\Delta \rangle_F.
\end{align*}
If we show, that $\langle \sgn(A)+\sgn(\Delta_{T^\perp}),\Delta \rangle_F>0$, it follows that $\|A+\Delta\|_*>\|A\|_*.$

We will show later that there is $Y\in\range ({\mathcal R})$ with
\begin{equation}\label{eq:G:12}
\|{\mathcal P}_TY-\sgn(A)\|_F\le \frac{1}{2n^2}\quad \text{and}\quad \|{\mathcal P}_{T^\perp}Y\|\le \frac{1}{2}.
\end{equation}
As $Y\in \range ({\mathcal R})$, we get $\langle Y,\Delta\rangle_F=\langle {\mathcal R}(\cdot),\Delta\rangle_F=\langle \cdot,{\mathcal R}(\Delta)\rangle_F=0$.
Then we finish this step by
\begin{align*}
\langle \sgn(A)+\sgn(\Delta_{T^\perp}),\Delta \rangle_F&=\langle \sgn(A)+\sgn(\Delta_{T^\perp})-Y,\Delta \rangle_F\\
&=\langle \sgn(A)-Y,\Delta_T \rangle_F+\langle \sgn(\Delta_{T^\perp})-Y,\Delta_{T^\perp} \rangle_F\\
&=\langle \sgn(\Delta_{T^\perp}),\Delta_{T^\perp} \rangle_F-\langle {\mathcal P}_{T^\perp}Y,\Delta_{T^\perp} \rangle_F-\langle {\mathcal P}_TY-\sgn(A),\Delta_T \rangle_F\\
&\ge \frac{1}{2}\|\Delta_{T^\perp}\|_*-\frac{1}{2n^2}\|\Delta_T\|_F \ge \frac{1}{2}\|\Delta_{T^\perp}\|F-\frac{1}{2n^2}\|\Delta_T\|_F>0.
\end{align*}

\noindent\underline{Step 4.: Existence of $Y\in \range\, {\mathcal R}$ with \eqref{eq:G:12}}\\

We present the proof only if \eqref{eq:G:20} holds and refer to \cite{G} for a proof under the condition \eqref{eq:G:21}.

We need to construct the \emph{dual certificate} $Y$ with the following properties
\begin{enumerate}
\item[(i)] $Y\in \range\, {\mathcal R}$,
\item[(ii)] $\displaystyle \|{\mathcal P}_TY-\sgn(A)\|_F\le \frac{1}{2n^2}$,
\item[(iii)] $\displaystyle \|{\mathcal P}_{T^\perp}Y\|\le \frac{1}{2}$.
\end{enumerate}
The most intuitive construction of $Y$ would be to take
$$
Y=\frac{n^2}{m}\sum_{i=1}^m\langle X_{\omega_i},\sgn(A) \rangle_F \cdot X_{\omega_i}
={\mathcal R}(\sgn(A)).
$$
Then (i) is clearly satisfied, and (ii) and (iii) hold for $\E Y=\sgn(A)$.
The hope is that application of concentration bounds on random matrices could give
the inequalities in (ii) and (iii).

Unfortunately, this construction of $Y$ does not converge quickly enough.
The \emph{golfing scheme} of \cite{G} constructs $Y$ in an iterative way.
Namely we put
$$
Y_1=\frac{n^2}{k}\sum_{i=1}^k\langle X_{\omega_i},\sgn(A) \rangle_F \cdot X_{\omega_i}.
$$
For good choice of $k$, $Y_1$ is already a reasonable approximation of $\sgn(A)$.
We then apply the same procedure to $\sgn(A)-{\mathcal P}_T Y_1$ and update the information
in this way, i.e. we put
$$
Y_2=Y_1+\frac{n^2}{k}\sum_{i=k+1}^{2k}\langle X_{\omega_i},\sgn(A)-{\mathcal P}_TY_1 \rangle_F \cdot X_{\omega_i}.
$$
The sequence ${\mathcal P}_T Y_i$ converges exponentially fast to $\sgn(A)$ in $l=m/k.$
On the other hand, we need to choose the $k$ large enough to allow for the application
of the concentration bounds.

To analyze the iterative scheme, we first need the following lemma.
\begin{lemma}\label{lem:G} Let $Z\in T$. Then
\begin{equation}
\P\Bigl(\|{\mathcal P}_{T^\perp}{\mathcal R} Z\|>t\Bigr)
\le \begin{cases}\displaystyle 2n\exp\Bigl(-\frac{t^2 m}{4\nu n\|Z\|_F^2}\Bigr)\quad \text{for}\quad
t\le \sqrt{2/r}\|Z\|_F,\\
\displaystyle 2n\exp\Bigl(-\frac{t m}{2\nu \sqrt{2r}n\|Z\|_F}\Bigr)\quad \text{for}\quad
t> \sqrt{2/r}\|Z\|_F.
\end{cases}
\end{equation}
\end{lemma}
\begin{proof}
It is enough to consider $\|Z\|_F=1.$ We put
$$
 S_{j}=\frac{n^2}{m}\langle X_j,Z\rangle_F{\mathcal P}_{T^\perp} X_j.
$$
Then
\begin{itemize}
\item $$\sum_{j=1}^m S_{\omega_j}={\mathcal P}_{T^\perp}{\mathcal R}Z;$$
\item $\E[S_{\omega_j}]=0$ due to (remember that $Z\in T$)
\begin{align*}
\E[S_{\omega_j}]&=\frac{n^2}{m}\cdot\frac{1}{n^2}\sum_{j=1}^{n^2} \langle X_j,Z\rangle_F{\mathcal P}_{T^\perp} X_j
={\mathcal P}_{T^\perp}\Bigl(\frac{1}{m}\sum_{j=1}^{n^2}\langle X_j,Z\rangle_F X_j\Bigr)
=\frac{1}{m}{\mathcal P}_{T^\perp}Z=0;
\end{align*}
\item the parameter $V_0^2$ is estimated by
\begin{align*}
\|\E[S_{\omega_j}^2]\| &= \Bigl\|\E\Bigl[S_{\omega_j}\circ S_{\omega_j}\Bigr]\Bigr\|\\
&= \Bigl\|\E\Bigl[\frac{n^4}{m^2}\langle X_{\omega_j},Z\rangle_F^2({\mathcal P}_{T^\perp} X_{\omega_j})^2\Bigr]\Bigr\|\\
&= \Bigl\|\frac{n^4}{m^2}\cdot \frac{1}{n^2}\sum_{j=1}^{n^2}\langle X_{j},Z\rangle_F^2({\mathcal P}_{T^\perp} X_{j})^2\Bigr\|\\
&\le \frac{n^2}{m^2}\sum_{j=1}^{n^2}\langle X_{j},Z\rangle_F^2\|({\mathcal P}_{T^\perp} X_{j})^2\|\\
&\le \frac{n^2}{m^2}\max_j \|({\mathcal P}_{T^\perp} X_{j})^2\|\cdot\sum_{j=1}^{n^2}\langle X_{j},Z\rangle_F^2\\
&\le \frac{n^2}{m^2}\frac{\nu}{n}\|Z\|_F^2=\frac{n\nu}{m^2}=:V_0^2;
\end{align*}
\item and finally
\begin{align*}
\|S_{\omega_j}\|&=\Bigl\|\frac{n^2}{m}\langle X_{\omega_j},Z\rangle_F{\mathcal P}_{T^\perp} X_{\omega_j}\Bigr\|\\
&\le \frac{n^2}{m} |\langle X_{\omega_j},Z\rangle_F|\cdot\|{\mathcal P}_{T^\perp} X_{\omega_j}\|\le \frac{n^2}{m}\sqrt{\frac{\nu}{n}} |\langle X_{\omega_j},Z\rangle_F|\\
&\le \frac{n^2}{m}\sqrt{\frac{\nu}{n}} \|X_{\omega_j}\|\cdot\|Z\|_*\le \frac{n^2}{m}\sqrt{\frac{\nu}{n}}\cdot \sqrt{\frac{\nu}{n}}\sqrt{2r}=
\frac{\nu n\sqrt{2r}}{m}=:c,
\end{align*}
as every $Z\in T$ has rank at most $2r$.
\item Observing that $\displaystyle 2mV_0^2/c=2m\cdot\frac{n\nu}{m^2}\cdot \frac{m}{\nu n\sqrt{2r}}=\sqrt{\frac{2}{r}}$,
the rest follows by Theorem \ref{theo:Bern1}.
\end{itemize}
\end{proof}
Let us now finish the proof of the existence of the dual certificate.
We split $m=m_1+\dots+m_l$ and define
\begin{equation}
{\mathcal R}_i:Z\to\frac{n^2}{m_i}\sum_{j=m_1+\dots+m_{i-1}+1}^{m_1+\dots+m_i} \langle X_{\omega_j},Z\rangle_F X_{\omega_j}.
\end{equation}
We set
\begin{align}
Y_0&=0,\quad Z_0=\sgn(A),\\
Y_i&=Y_{i-1}+{\mathcal R}_i X_{i-1}=\sum_{j=1}^i {\mathcal R}_jZ_{j-1},\\
Z_i&=\sgn(A)-{\mathcal P}_TY_i.
\end{align}
We get
\begin{align*}
Z_0&=\sgn(A),\\
Z_1&=\sgn(A)-{\mathcal P}_TY_1=\sgn(A)-{\mathcal P}_T{\mathcal R}_1\sgn(A)=(Id-{\mathcal P}_T{\mathcal R}_1{\mathcal P}_T)\sgn(A),\\
Z_2&=\sgn(A)-{\mathcal P}_TY_2=\sgn(A)-{\mathcal P}_T\Bigl(Y_1+{\mathcal R}_2Z_1\Bigr)\\
&=\sgn(A)-{\mathcal P}_T\Bigl({\mathcal R}_1{\mathcal P}_T\sgn(A)+{\mathcal R}_2(Id-{\mathcal P}_T{\mathcal R}_1{\mathcal P}_T)\sgn(A)\Bigr)\\
&=(Id-{\mathcal P}_T{\mathcal R}_1{\mathcal P}_T)\sgn(A)-{\mathcal P}_T{\mathcal R}_2(Id-{\mathcal P}_T{\mathcal R}_1{\mathcal P}_T)\sgn(A)\\
&=(Id-{\mathcal P}_T{\mathcal R}_2{\mathcal P}_T)(Id-{\mathcal P}_T{\mathcal R}_1{\mathcal P}_T)\sgn(A)
=(Id-{\mathcal P}_T{\mathcal R}_2{\mathcal P}_T)Z_1,\\
\vdots&\\
Z_i&=(Id-{\mathcal P}_T{\mathcal R}_i{\mathcal P}_T)(Id-{\mathcal P}_T{\mathcal R}_{i-1}{\mathcal P}_T)\dots(Id-{\mathcal P}_T{\mathcal R}_1{\mathcal P}_T)\sgn(A).
\end{align*}
Assume that (with probability of failure at most $p_2(i)$)
$$
\|Z_i\|_F=\|(Id-{\mathcal P}_T{\mathcal R}_i{\mathcal P}_T)Z_{i-1}\|_F=\|({\mathcal P}_T-{\mathcal P}_T{\mathcal R}_i{\mathcal P}_T)Z_{i-1}\|_F\le \frac{1}{2}\|Z_{i-1}\|_F.
$$
Then
$$
\|Z_i\|_2\le \frac{\sqrt{r}}{2^i}.
$$
Furthermore, we assume that (with the probability of failure at most $p_3(i)$)
$$
\|{\mathcal P}_{T^\perp}{\mathcal R}_iZ_{i-1}\|\le \frac{1}{4\sqrt{r}}\|Z_{i-1}\|_F,
$$
which gives
\begin{align*}
\|{\mathcal P}_{T^\perp}Y_l\|&\le \sum_{i=1}^l\|{\mathcal P}_{T^\perp}{\mathcal R}_iZ_{i-1}\|\le \frac{1}{4\sqrt{r}}\sum_{i=1}^l\|Z_{i-1}\|_F\\
&\le \frac{1}{4\sqrt{r}}\sum_{i=1}^l\frac{\sqrt{r}}{2^{i-1}}<\frac{1}{4}\sum_{i=0}^\infty \frac{1}{2^i}=\frac{1}{2}.
\end{align*}
and
\begin{equation}
\|Z_l\|=\|{\mathcal P}_TY_l-\sgn(A)\|\le\frac{\sqrt{r}}{2^l}\le\frac{1}{2n^2}
\end{equation}
for $l=\lceil \log_2(2n^2\sqrt{r})\rceil$.

Finally, we have to estimate the probabilities $p_1, p_2(i)$ and $p_3(i)$ and ensure that
$$
p_1+\sum_{i=1}^lp_2(i)+\sum_{i=1}^lp_3(i)\le n^{-\beta}.
$$
Recall that
$$
p_1=4nr\exp\Bigl(-\frac{m}{16(2\nu rn+1)}\Bigr).
$$

By \eqref{eq:G'} and using that $Z_i\in T$ we get
$$
\P(\|{\mathcal P}_T-{\mathcal P}_T{\mathcal R}_i{\mathcal P}_T\|\ge 1/2)\le 4nr\exp\Bigl(-\frac{m_i}{16(2\nu rn+1)}\Bigr)=:p_2(i).
$$
Furthermore, Lemma \ref{lem:G} gives
\begin{align*}
\P\Bigl(\|{\mathcal P}_{T^\perp}{\mathcal R} Z_{i-1}\|>\frac{\|Z_{i-1}\|_F}{4\sqrt{r}}\Bigr)
&\le 2n\exp\Bigl(-\frac{\|Z_{i-1}\|^2_F m_i}{16r\cdot 4\nu n\|Z_{i-1}\|_F^2}\Bigr)\\
&=2n\exp\Bigl(-\frac{m_i}{64\nu rn}\Bigr)=:p_3(i).
\end{align*}
Here, we have used that $\displaystyle t=\frac{\|Z_{i-1}\|_F}{4\sqrt{r}}< \frac{\sqrt{2}\|Z_{i-1}\|_F}{\sqrt{r}}.$
Finally, to ensure that $p_2(i)$ and $p_3(i)$ are both bounded by $n^{-\beta}/3$, it is enough to chose
$$
m_i\ge 64\nu rn[\ln(6nr)+\ln(2l)+\beta\ln(n)]
$$
leading to
$$
m=\sum_{i=1}^l m_i\ge 64l\nu rn[\ln(6nr)+\ln(2l)+\beta\ln(n)]=O(\nu rn(1+\beta)\ln^2 n).
$$
\end{proof}

\newpage

\thispagestyle{scrplain}

\end{document}

%% file: fig_c.pstex_t
\begin{picture}(0,0)%
\includegraphics{fig_c.pstex}%
\end{picture}%
\setlength{\unitlength}{4144sp}%
\begingroup\makeatletter\ifx\SetFigFont\undefined%
\gdef\SetFigFont#1#2#3#4#5{%
  \reset@font\fontsize{#1}{#2pt}%
  \fontfamily{#3}\fontseries{#4}\fontshape{#5}%
  \selectfont}%
\fi\endgroup%
\begin{picture}(4527,3354)(211,-4168)
\put(451,-3841){\makebox(0,0)[lb]{\smash{{\SetFigFont{12}{14.4}{\familydefault}{\mddefault}{\updefault}{\color[rgb]{0,0,0}$O$}%
}}}}
\put(4591,-3841){\makebox(0,0)[lb]{\smash{{\SetFigFont{12}{14.4}{\familydefault}{\mddefault}{\updefault}{\color[rgb]{0,0,0}$u$}%
}}}}
\put(1756,-1906){\makebox(0,0)[lb]{\smash{{\SetFigFont{12}{14.4}{\familydefault}{\mddefault}{\updefault}{\color[rgb]{0,0,0}$P$}%
}}}}
\put(4321,-3481){\makebox(0,0)[lb]{\smash{{\SetFigFont{12}{14.4}{\familydefault}{\mddefault}{\updefault}{\color[rgb]{0,0,0}$A$}%
}}}}
\put(766,-1141){\makebox(0,0)[lb]{\smash{{\SetFigFont{12}{14.4}{\familydefault}{\mddefault}{\updefault}{\color[rgb]{0,0,0}$B$}%
}}}}
\put(4051,-3886){\makebox(0,0)[lb]{\smash{{\SetFigFont{12}{14.4}{\familydefault}{\mddefault}{\updefault}{\color[rgb]{0,0,0}$t/\lambda_1$}%
}}}}
\put(226,-1411){\makebox(0,0)[lb]{\smash{{\SetFigFont{12}{14.4}{\familydefault}{\mddefault}{\updefault}{\color[rgb]{0,0,0}$t/\lambda_2$}%
}}}}
\put(2701,-2491){\makebox(0,0)[lb]{\smash{{\SetFigFont{12}{14.4}{\familydefault}{\mddefault}{\updefault}{\color[rgb]{0,0,0}$\lambda_1u+\lambda_2v=t$}%
}}}}
\put(496,-961){\makebox(0,0)[lb]{\smash{{\SetFigFont{12}{14.4}{\familydefault}{\mddefault}{\updefault}{\color[rgb]{0,0,0}$v$}%
}}}}
\end{picture}%